\documentclass[a4paper, 12pt]{amsart}
\usepackage{amssymb, amsmath, amsthm, enumerate, color, subcaption, tipa, upgreek, xcolor}
\usepackage{amstext, tikz, appendix, mathrsfs, mathabx, mathtools, pgfplots, stmaryrd}
\usepackage{csquotes}
\usepackage[colorlinks=true, breaklinks=true, linkcolor=black, citecolor=blue, urlcolor=red]{hyperref} 
\usepackage[english]{babel}
\usepackage{forest, adjustbox}
\usepackage{tikz-cd}
\usepackage{enumitem}

\numberwithin{equation}{section}
\newtheorem{letterthm}{Theorem}

\newtheorem{lettercor}[letterthm]{Corollary}

\newtheorem{letterdefinition}[letterthm]{Definition}

\newtheorem{theorem}{Theorem}[section]

\newtheorem{lemma}[theorem]{Lemma}
\newtheorem{corollary}[theorem]{Corollary}
\newtheorem{proposition}[theorem]{Proposition}

\newtheorem{definition}[theorem]{Definition}
\newtheorem{notation}[theorem]{Notation}
\newtheorem{remark}[theorem]{Remark}
\newtheorem{example}[theorem]{Example}

\newcommand{\act}{\curvearrowright}

\newcommand{\Arg}{{\text{Arg}}}
\DeclareMathOperator{\atom}{atom}

\newcommand{\bone}{\mathbf 1}

\newcommand{\cC}{\mathcal C}

\newcommand{\C}{\mathbf C}

\DeclareMathOperator{\comp}{comp}

\newcommand{\cD}{\mathcal D}

\DeclareMathOperator{\diff}{diff}

\DeclareMathOperator{\End}{End}

\newcommand{\F}{\mathbf F}

\newcommand{\fH}{\mathfrak H}

\newcommand{\scrH}{\mathscr H}
\DeclareMathOperator{\Hilb}{Hilb}

\DeclareMathOperator{\Hom}{Hom}

\newcommand{\cI}{\mathcal I}

\DeclareMathOperator{\id}{id}
\DeclareMathOperator{\Ind}{Ind}

\DeclareMathOperator{\Irr}{Irr}

\newcommand{\into}{\hookrightarrow}
\newcommand{\cJ}{\mathcal J}

\newcommand{\fK}{\mathfrak K}
\newcommand{\scrK}{\mathscr K}
\newcommand{\fL}{\mathfrak L}

\newcommand{\scrM}{\mathscr M}
\DeclareMathOperator{\Mod}{Mod}

\newcommand{\N}{\mathbf{N}}

\newcommand{\NI}{\mathscr M}
\newcommand{\NII}{\mathscr N}
\newcommand{\NInd}{\text{Ind-mixing}}

\newcommand{\cO}{\mathcal O}
\DeclareMathOperator{\ob}{ob}

\newcommand{\ot}{\otimes}

\newcommand{\cP}{\mathcal P}
\newcommand{\scrP}{\mathscr P}

\newcommand{\R}{\mathbf{R}}

\DeclareMathOperator{\Rep}{Rep}

\DeclareMathOperator{\res}{res}
\DeclareMathOperator{\FD}{FD}
\DeclareMathOperator{\full}{full}

\DeclareMathOperator{\supp}{supp}

\newcommand{\ti}{\tilde}

\newcommand{\cU}{\mathcal U}

\newcommand{\varep}{\varepsilon}

\newcommand{\fZ}{\mathfrak Z}

\DeclarePairedDelimiterX{\norm}[1]{\lVert}{\rVert}{#1}

\providecommand{\keywords}[1]{\textbf{\textit{Index terms---}} #1}

\begin{document}

	\title[Jones' representations]{A tensor product for representations of the Cuntz algebra and of the R.~Thompson groups}
	\thanks{
		AB is supported by the Australian Research Council Grant DP200100067.\\
		DW is supported by an Australian Government Research Training Program (RTP) Scholarship.}
	\author{Arnaud Brothier and Dilshan Wijesena}
	\address{Arnaud Brothier\\ University of Trieste, Via Weiss, 2 34128 Trieste, Italia and
	School of Mathematics and Statistics, University of New South Wales, Sydney NSW 2052, Australia}
	\email{arnaud.brothier@gmail.com\endgraf
		\url{https://sites.google.com/site/arnaudbrothier/}}
\address{Dilshan Wijesena\\ School of Mathematics and Statistics, University of New South Wales, Sydney NSW 2052, Australia}
	\email{dilshan.wijesena@hotmail.com}
	\maketitle

	\begin{abstract}
The authors continue a series of articles studying certain unitary representations of the Richard Thompson groups $F,T,V$ called Pythagorean. They all extend to the Cuntz algebra $\cO$ and conversely all representations of $\cO$ are of this form. Via this approach we introduce a tensor product for a large class of representations of $F,T,V,\cO$. 
We prove that a sub-category forms a tensor category and perform a number of explicit computations of fusion rules.

	\end{abstract}
	
\keywords{{\bf Keywords:} Richard Thompson's groups, Cuntz algebra, Jones' representations, tensor category.}


\section*{Introduction}

We continue to study Pythagorean representations (in short P-representations) initiated in \cite{Brothier-Jones19,Brothier-Wijesena22,Brothier-Wijesena24,Brothier-Wijesena23}.
Vaughan Jones developed a powerful tool, termed \textit{Jones' technology}, to diagrammatically construct actions of groups \cite{Jones17}.
In particular, from any (linear) isometry $R:\fH\to\fH\oplus\fH$ between Hilbert spaces we can define {\it unitary} representations of the Richard Thompson groups $F,T,V$ \cite{Brothier-Jones19}.
This is a functorial process consisting in enlarging $(R,\fH)$ into $(\Pi(R),\scrH)$ so that $\Pi(R):\scrH\to\scrH\oplus\scrH$ is a {\it surjective} isometry.
The groups $F,T,V$ act on $\scrH$: these are the P-representations.

If $A,B$ are the legs of $R$, then the isometric condition translates into $A^*A+B^*B=\id$. 
This relation defines a universal $C^*$-algebra $\cP:=\cP_2$, with generators denoted $a,b$, called the Pythagorean algebra. 
Hence, any representation $(A,B,\fH)$ of $\cP$ produces a P-representation of the Thompson groups $F,T,V$.
Similarly, the legs of $\Pi(R)$ define a representation of the Cuntz(--Dixmier) algebra $\cO:=\cO_2$, i.e.~the universal $C^*$-algebra generated by $s_0,s_1$ satisfying $s_0s_0^*+s_1s_1^*=1$ and additionally $s_0^*s_0=s_1^*s_1=1$ \cite{Cuntz77,Dixmier64}.
These constructions are functorial giving four Pythagorean functors $\Pi^X:\Rep(\cP)\to\Rep(X)$ (in short P-functors) and are compatible with the Birget--Nekrashevych embedding $V\into\cO$ so that $\Pi^V(R)=\Pi^\cO(R)\restriction_V$, see Section \ref{sec:Cuntz} and \cite{Birget04,Nekrashevych04}.
All representations of $\cO$ arise in that way but not all representations of $F,T,V$ are Pythagorean, see \cite{Brothier-Jones19}.

Studying $R$ to deduce properties of $\Pi^X(R)$ has proven to be extremely powerful.
Though this has a cost: we must add morphisms inside $\Rep(\cP)$ to expect any classification results. We then obtain the category of {\it P-modules} $\Mod(\cP)$ with the same objects as $\Rep(\cP)$; however, morphisms being now bounded linear maps intertwinning the $A$'s and $B$'s but not necessarily their adjoints. 
This makes the representation theory of $\cP$ more subtle. 
A finite-dimensional P-module $m=(A,B,\fH)$ decomposes uniquely into $\fH=\fH_{\comp}\oplus\fH_{\res}$ where the {\it residual subspace} $\fH_{\res}$ is the largest vector subspace not containing any non-trivial subspace closed under $A,B$. 
Then $\fH_{\comp}\subset\fH$ is the {\it complete P-submodule} satisfying $\Pi^X(\fH_{\comp})\simeq \Pi^X(\fH)$. 
Remarkably, the usual dimension $\dim(\fH_{\comp})$ is an invariant of both the P-module $\fH$ and the representation $\Pi^X(\fH)$ (for $X=F,T,V,\cO$) called the {\it Pythagorean dimension} (in short P-dimension) denoted $\dim_P(\fH)=\dim_P(\Pi^X(\fH)).$
Further, $\fH_{\comp}$ decomposes into a finite direct sum of irreducible P-submodules. These latter are either {\it atomic} or {\it diffuse}. The first kind yields monomial representations of $F,T,V$ while the second are Ind-mixing representation (i.e.~they do not contain the induction of a finite-dimensional representation).
Additionally, $\Pi^X$ is fully faithful (up to few exceptions) when restricted to finite-dimensional full P-modules (full means that $\fH_{\res}=\{0\}$). This latter is a consequence of the main results of \cite{Brothier-Wijesena24,Brothier-Wijesena23}. 
In particular, the P-functors send distinct irreducible P-modules to distinct irreducible P-representations.
From there we deduce moduli spaces of P-modules and P-representations. 
Consequently, a large part of the spectrum of $\cO$ has a beautiful geometric structure. This is a complete surprise knowing that $\cO$ is a simple non-type $I$ $C^*$-algebra and thus a priori has a highly pathological representation theory by Glimm's theorem \cite{Glimm61}.

\textbf{Aim of this article.} We have demonstrated that a large class of representations of $F,T,V,\cO$ are very well-behaved. 
We now wish to enrich this picture with {\it fusion}, i.e.~having a tensor product for P-representations.
This is the main topic of this article: we define a novel monoidal product for sub-classes of P-modules and P-representations, and perform a number of explicit computations.

This is a non-trivial task. Indeed, the classical tensor product of Pythagorean representations of the Thompson groups are not Pythagorean. Moreover, there is no theories of tensor product of representations of $C^*$-algebras. 
It is known that $\cO\simeq \cO\otimes \cO$ which defines a tensor product on $\Rep(\cO)$. However, $\cO\simeq \cO\otimes \cO$ being obtained by abstract classification, the resulting tensor product is rather intractable. There is an interesting binary operation studied by Kawamura $\Rep(\cO_n)\times \Rep(\cO_m)\to \Rep(\cO_{nm})$ that we briefly present in Section \ref{sec:Kawamura} using the Pythagorean algebra. However, this is not a tensor product for representations of the single Cuntz algebra $\cO_2$.

Our strategy is to define a novel tensor product $\boxtimes$ (in opposition with the classical tensor product $\otimes$) using $\cP$ and then to push it onto $F,T,V,\cO$ using the P-functors $\Pi^X$.
\begin{letterdefinition} \label{letterdef:tensor-prod-def}
	Let $(A,B,\fH), (\ti A, \ti B, \ti \fH)$ be any two P-modules satisfying
	\[\ker(A \otimes \ti A) \cap \ker(B \otimes \ti B) = \{0\}.\]
	Define the following binary operation $\boxtimes$ that produces another P-module:
		\[(A,B,\fH) \boxtimes (\ti A, \ti B, \ti \fH) = ((A \otimes \ti A)K^{-1}, (B \otimes \ti B)K^{-1}, \fH \otimes \ti \fH) \text{ where } \]	
	\[K := \sqrt{\vert A \vert^2 \otimes \vert \ti A \vert^2 + \vert B \vert^2 \otimes \vert \ti B \vert^2}.\]	
\end{letterdefinition}
We could na\"ively want to directly use this formula at the level of $\cO$ (rather than $\cP$). This always fail as $A,B$ would necessarily be partial isometries with orthogonal domain making the formula of $\boxtimes$ ill-defined.
Definition \ref{letterdef:tensor-prod-def} makes sense for any dimension (using unbounded operators and carefully defining $K^{-1}$). 
However, we have chosen to restrict this article to finite-dimensional $\fH$ where the analysis remains elementary and where our techniques are more powerful.

\begin{center}\textbf{For the remainder of the introduction, all P-modules are finite-dimensional and all P-representations have finite P-dimension.}
\end{center}

The operation $\boxtimes$ preserves equivalence classes of P-modules. It is distributive with respect to the direct sum. Moreover, if we fix co-ordinates and consider P-modules as matrices, then $\boxtimes$ defines smooth operations between manifolds. 
However, it does not preserve {\it fullness} of P-modules in general.
This means that $\boxtimes$ may create residual subspaces. This causes a number of problems such as not being multiplicative for the P-dimension, and not preserving equivalence classes of P-representations once pushed to $\Rep(X)$ for $X=F,T,V,\cO$. In order to obtain a suitable monoidal structure at the level of P-representations we must find a class of \emph{full} P-modules that is closed under $\boxtimes$ and on which the P-functors $\Pi^X:\Mod(\cP)\to\Rep(X)$ are fully faithful.
We suggest two such classes (which may be not optimal, see Remark \ref{rem:tensor-cat}) and observe that the restriction of $\boxtimes$ to them satisfy a number of remarkable properties. We follow here the terminology of \cite{Etingof-Gelaki-Nikshych-Ostrik16} and refer to Section \ref{subsec:intro-tensor-cat} for details.

\begin{letterthm}\label{letterthm:tensor}(Proposition \ref{prop:P-rep-tensor-prod} and Theorem \ref{thm:tensor-category})
Let $\NI$ (resp.~$\NII$) be the class of finite-dimensional P-modules $(A,B,\fH)$ where $A$ is normal (resp.~$A$ is normal and invertible) and $B$ is invertible.
The following statements hold:
\begin{enumerate}
\item $\NI$ and $\NII$ are closed under $\boxtimes$ and are contained in the class of full P-modules;
\item $(\NI, \boxtimes)$ is a $\C$-linear symmetric monoidal category;
\item if $m,\ti m\in \NI$, then $\dim_P(m\boxtimes \ti m)=\dim_P(m)\cdot \dim_P(\ti m)$;
\item $(\NII, \boxtimes)$ is a pivotal unitary tensor category;
\item the quantum dimension of $(\NII, \boxtimes)$ coincides with the P-dimension.
\end{enumerate}
\end{letterthm}
The monoidal category $(\NII,\boxtimes)$ satisfies all the axioms of a unitary fusion category except that it has infinitely many classes of simple objects (and in fact continuously many). 

We now transport $\boxtimes$ from {\it P-modules} to {\it P-representations} of $X=F,T,V,\cO$ using an equivalence of categories.
Let $\Rep^{\FD}(Y)$, $Y=T,V,\cO$ (resp.~$\Rep^{\FD+}(F)$) be the full sub-category of representation of finite P-dimension (resp.~finite P-dimension and does not contain an atomic representation of dimension one). 
We then define $M^Y:\Rep^{\FD}(Y)\to \Mod(\cP)$ and $\Rep^{\FD+}(F)\to \Mod(\cP)$ that consists in extending the $Y$ or $F$-representation into an $\cO$-representation and then to extract from the $\cO$-representation the smallest complete P-submodule. This is well-defined thanks to the main results of \cite{Brothier-Wijesena24,Brothier-Wijesena23}.

\begin{letterdefinition}\label{letterdef:P-rep}
Consider $X=F,T,V,\cO$ and two P-representations $\pi,\ti\pi$ of $X$ in $\Rep^{\FD+}(F)$ if $X=F$ and in $\Rep^{\FD}(X)$ otherwise.
If $M^X(\pi),M^X(\ti\pi)$ fulfill the assumption of Definition \ref{letterdef:tensor-prod-def}, then we define 
$$\pi\boxtimes \ti \pi:=\Pi^X(M^X(\pi)\boxtimes M^X(\ti\pi)).$$
\end{letterdefinition}

We now further restrict the domain of $M^X$ so that it is closed under $\boxtimes$ obtaining monoidal categories of \emph{P-representations} (see Section \ref{subsec:tens-cat-p-rep}).

\begin{lettercor}\label{lettercor:tensor}
Given $X=F,T,V,\cO$ we write $\Rep_{\NI}(X)$ for the full sub-category with class of objects $\Pi^X(m)$ with $m\in \NI$ (and similarly define $\Rep_{\NII}(X)$).
The pair of functors 
$$(\Pi^F,M^F) \text{ between $\NII$ and $\Rep_{\NII}(F)$}$$
defines an equivalence of categories.
The binary operation $\boxtimes$ of Definition \ref{letterdef:P-rep} provides a monoidal structure on $\Rep_{\NII}(X)$ monoidally equivalent to $(\NII,\boxtimes)$ and thus inherits all the properties enunciated in Theorem \ref{letterthm:tensor}.

A similar statement holds after replacing $F$ by $T,V$ or $\cO$, and $\NII$ with $\NI$ or $\NII$.
\end{lettercor}

We consider explicit examples and computations of fusions. 
To do this we use co-ordinates so that P-modules are now matrices. 
We start by considering the dimension one case obtaining a Lie group structure on irreducible classes of diffuse P-representations of P-dimension one. We moreover deduce Lie group actions of moduli spaces of P-modules and P-representations.
Then we compute explicit fusion rules for a number of higher dimensional P-modules and P-representations such as the \emph{generalised permutation} representations of Kawamura \cite{Kawamura05}.

\textbf{Plan of the article.}
Section \ref{sec:prelim} recalls classical facts on $F,T,V$ and $\cO$. Moreover, we recall previous results on the Pythagorean algebra and on P-representations. Finally, we introduce some category terminology.
Section \ref{sec:tensor-product} is the most important where we introduce Definitions \ref{letterdef:tensor-prod-def} and \ref{letterdef:P-rep}, and prove Theorem \ref{letterthm:tensor} and Corollary \ref{lettercor:tensor}.
Section \ref{sec:examples} provides many explicit examples.
Section \ref{sec:Kawamura} presents a previous binary operation on representations of all the Cuntz algebras and re-interpret it in our framework.

\textbf{Acknowledgments.}
The authors warmly thanks Pinhas Grossman for answering questions and providing insightful discussions regarding tensor categories.


\section{Preliminaries} \label{sec:prelim}

We briefly recall classical materials concerning the Richard Thompson groups, the Cuntz(--Dixmier) algebra, the connection between those, and some more recent materials on representations of the Pythagorean algebra.
For extensive details we refer the reader to \cite{Cannon-Floyd-Parry96} for the Thompson groups $F,T,V$, to \cite{Cuntz77,Aiello-Conti-Rossi21} for the Cuntz algebra $\cO$, to \cite{Birget04, Nekrashevych04} for the connection between $F,T,V$ and $\cO$, and to our previous papers \cite{Brothier-Jones19, Brothier-Wijesena22, Brothier-Wijesena24, Brothier-Wijesena23} for materials concerning Pythagorean representations.

{\bf Conventions.} 
The set $\N$ of natural numbers contains $0$ and write $\N^*$ for $\N\setminus\{0\}.$
All Hilbert spaces (often denoted $\fH$ or $\scrH$) are assumed to be over $\C$ with inner-product linear in the first variable.

\subsection{The Richard Thompson groups}
Richard Thompson's group $F$ is the group of homeomorphism on the Cantor space $\cC := \{0,1\}^{\N^*}$ that acts by locally changing finite prefixes and preserving the lexicographic order of $\cC$.
A finite binary word $w$ defines a \emph{cylinder set} $I_w := \{w \cdot x : x \in \cC\} \subset \cC$. 
An element $g \in F$ is determined by two ordered standard partitions of $\cC$ of equal length, $\mathcal I=(I_{v_1}, \dots, I_{v_n})$ and $\mathcal J=(I_{w_1}, \dots, I_{w_n})$, such that $g(v_j \cdot x) = w_j \cdot x$ for all $x \in \cC$. Conversely, any such pair $(\cI,\cJ)$ defines an element of $F$.
Thompson's group $T$ (resp.~$V$) is similarly defined with elements determined by $(\cI,\cJ,\kappa)$ with mapping $g(v_j\cdot x) = w_{\kappa(j)}\cdot x$ with $\kappa$ a cyclic permutation (resp.~a permutation).

\subsection{The Cuntz algebra}\label{sec:Cuntz}
Recall, the \textit{Cuntz algebra} in the universal $C^*$-algebra $\cO := \cO_2$ with two generators $s_0$, $s_1$ satisfying the relations
\[s_0^*s_0 = s_1^*s_1 = s_0s_0^* + s_1s_1^* = 1.\]
By ``universal'' property of $\cO$, all representations of $\cO$ are given by triples $(S_0, S_1, \fH)$ where $S_0$, $S_1$ are bounded linear isometries on a complex Hilbert space $\fH$ such that
\[S_0S_0^* + S_1S_1^* = \id_{\fH}.\]
with $s_i \mapsto S_i$ for $i=0,1$.

\textbf{Embedding Thompson's group $V$ inside $\cO$.}
Birget and Nekrashevych independently made the remarkable discovery that Thompson's group $V$ (and thus $F$ and $T$) embeds inside the unitary group of $\cO$, $\cU(\cO) := \{u \in \cO : u^*u = uu^* = 1\}$ \cite{Birget04, Nekrashevych04}. Specifically, the map
\[\iota : V \rightarrow \cU(\cO) : g \mapsto \sum_{i=1}^n s_{w_i}s_{v_i}^*\]
defines a group homomorphism of $V$ into $\cU(\cO)$ where $g$ is the element in $V$ that maps the cylinder set $I_{v_i}$ to $I_{w_i}$. 
Here, $s_w=s_{i_1}\dots s_{i_k}$ if $w=i_1\dots i_n$ is a binary word.

\subsection{Pythagorean modules}
\subsubsection{The Pythagorean algebra}
The \textit{Pythagorean algebra} $\cP$ is the universal $C^*$-algebra with two generators $a$, $b$ which satisfy the so-called Pythagorean equality
\[a^*a + b^*b = 1.\]
The Cuntz algebra is a quotient of the Pythagorean algebra via the mapping $a\mapsto s_0^*, b\mapsto s_1^*$.
Representations of the $C^*$-algebra $\cP$ are triples $(A,B,\fH)$ such that $A,B \in B(\fH)$ and $A^*A+B^*B = \id_\fH$ where $\fH$ is a Hilbert space. 
We consider intertwinners between such representations which forget the *-structure. 
This forms a category $\Mod(\cP)$ with same class of object as the usual representation category $\Rep(\cP)$ but having more morphisms. 
To avoid confusions we use the term ``module" rather than ``representation".

\begin{definition}
\begin{enumerate}
\item A \textit{Pythagorean module} (in short P-module) is a triple $(A,B,\fH)$ where $A,B \in B(\fH)$, $\fH$ is a Hilbert space and $A,B$ satisfies
\[A^*A + B^*B = \id_{\fH}.\] 
\item A \textit{P-submodule} of $(A,B,\fH)$ is a closed subspace $\fK \subset \fH$ such that $A(\fK), B(\fK) \subset \fK$. 
\item An \textit{intertwinner} between two P-modules $m=(A,B,\fH)$, $\ti m=(\ti A, \ti B, \ti \fH)$ is a bounded linear map $\theta : \fH \rightarrow \ti \fH$ such that 
\[\ti A \circ \theta = \theta \circ A \textrm{ and } \ti B \circ \theta = \theta \circ B.\] 
\item If $\theta$ is a unitary (i.e.~$\theta\circ \theta^*=\id_{\ti \fH}$ and $\theta^*\circ\theta=\id_\fH$) then we say that $m$ and $\ti m$ are {\it unitary equivalent} or simply {\it equivalent}.
\item Write $\Mod(\cP)$ for the category of P-modules with intertwinners for morphisms. 
\end{enumerate}
\end{definition}

We refer the reader to \cite[Section 6]{Brothier-Jones19} and \cite[Section 6]{Brothier-Wijesena23} for examples.

\begin{definition} \label{def:P-mod-def}
	For a P-module $m = (A,B,\fH)$ so that $\fH = \fH_0 \oplus \fZ$ where $\fH_0$ is a P-submodule and $\fZ$ is a subspace not containing any non-trivial P-submodules we say:
	\begin{enumerate}
		\item $\fH_0$ is a \textit{complete} P-submodule of $\fH$;
		\item $\fZ$ is a \textit{residual} subspace of $\fH$;
		\item $\fH$ is \textit{full} if $\fH$ is the only complete P-submodule of $\fH$;
		\item $\fH$ is \textit{irreducible} if $\fH$ does not contain any non-trivial proper P-submodules, otherwise $\fH$ is reducible.
	\end{enumerate}
\end{definition}

\begin{center}
	\textbf{For the remainder of this article we shall assume that all P-modules are \textit{finite-dimensional} unless specified otherwise.}	
\end{center}

\subsubsection{Diffuse, atomic, Pythagorean dimension}\label{sec:decomposition-irreducible}
Fix a P-module $m=(A,B,\fH)$.

{\bf Decomposition of P-modules.}
We have a decomposition 
$$\fH=\fH_{\comp}\oplus \fH_{\res} = \oplus_{j=1}^k \fH_j \oplus \fH_{\res}$$
where $\fH_{\comp}$ is the smallest complete P-submodule of $\fH$, $\fH_{\res}$ is the largest residual subspace of $\fH$, and each $\fH_j$ is an irreducible P-submodule.
We call $\fH_{\comp}$ (resp.~$\fH_{\res}$) \textit{the} complete P-submodule (resp.~\textit{the} residual subspace) of $\fH$.

{\bf Rays and infinite words.}
We interpret elements of the cantor space $\cC$ as \textit{rays}, i.e.~elements of the boundary of the rooted infinite complete binary tree.
Given a ray $p$ we write $p_n$ for the word made of the first $n$ digits of $p$.
The word $p_n$ defines an operator acting on any P-module $m=(A,B,\fH)$ by replacing $0$'s and $1$'s by $A$'s and $B$'s, and by reversing the order (e.g. if $p = 011\dots$ then $p_3$ gives the operator $BBA$).

{\bf Diffuse and atomic P-modules.}
Define $\fH_{\diff}$ to be the subspace of all vectors $\xi \in \fH_{\comp}$ such that $p_n\xi$ tends to $0$ for all rays $p \in \cC$. 
Additionally, define $\fH_{\atom}$ to be the span of all vectors $\eta \in \fH_{\comp}$ such that there exists a ray $p \in \cC$ satisfying $\norm{p_n\eta} = \norm{\eta}$ for all $n \in \N$. 
It is easy to verify that $\fH_{\diff}$ and $\fH_{\atom}$ are P-submodules of $\fH_{\comp}$ which we call the \textit{diffuse} and \textit{atomic} parts, respectively, of $\fH$. Moreover, we have the decomposition:
\[\fH_{\comp} = \fH_{\atom} \oplus \fH_{\diff}.\]

{\bf Pythagorean dimension of a P-module.}
The {\it Pythagorean dimension} (in short P-dimension) $\dim_P(\fH)$ of the P-module $m=(A,B,\fH)$ is equal to the usual dimension $\dim(\fH_{\comp})$ of the complex vector space $\fH_{\comp}$.

\subsection{Pythagorean representations of Thompson's groups and Cuntz's algebra} \label{subsec:p-rep-prelim}
\subsubsection{The Pythagorean functors}
Given a P-module $m=(A,B,\fH)$ there is a systematic construction of a Hilbert space $\scrH$ and two partial isometries $\tau_0,\tau_1$ so that $m$ is a P-submodule of $(\tau_0,\tau_1,\scrH)$ \cite{Brothier-Jones19}.
Moreover, now the map $\xi\mapsto \tau_0\xi\oplus \tau_1\xi$ is surjective and thus is a unitary transformation from $\scrH$ to $\scrH\oplus\scrH$.
We deduce that $(\tau_0,\tau_1,\scrH)$ defines a representation of the Cuntz algebra $\cO$ via: $s_i\mapsto \tau_i^*, i=0,1$.
Using the Birget--Nekrashevych embedding $V\into\cO$ we obtain unitary representations of the Thompson groups $F,T,V$.
This process is functorial and defines four functors:
$$\Pi^X:\Mod(\cP)\to \Rep(X) \text{ for } X=F,T,V,\cO$$
where $\Rep(X)$ is the usual category of unitary representations (resp.~*-representations) for $X = F,T,V$ (resp.~$X = \cO$).
They are called {\it Pythagorean functors} (in short P-functors) and $\Pi^X(m)$ is called a Pythagorean representation (in short P-representation) of $X$. Surprisingly, \textit{all} representations of $\cO$ are Pythagorean (if allowing infinite-dimensional P-modules), that is, $\Pi^\cO$ is essentially surjective on objects, see \cite[Proposition 7.1]{Brothier-Jones19}.

Even though all representations of $\cO$ are Pythagorean we may refer to a representation $\pi$ of $\cO$ as being a P-representation to emphasis that we view $\pi$ as $\Pi^X(m)$ for a certain P-module $m$.
Lastly, for convenience we may write $\Pi^X(\fH)$ in place for $\Pi^X(m)$. 

To conclude, below is an important result which motivates the terminology of \textit{complete} P-submodules and \textit{residual} subspaces, and allows to only consider \textit{full} P-modules.

\begin{proposition} \cite[Proposition 2.4]{Brothier-Wijesena24}
	Let $\fH = \fH_{\comp} \oplus \fH_{res}$ for a finite-dimensional P-module ($A,B,\fH)$. Then $\Pi^X(\fH) \cong \Pi^X(\fH_{\comp})$ for $X = F,T,V,\cO$.
\end{proposition}

Subsequently, we define the \textit{P-dimension} $\dim_P(\sigma)$ of a P-representation $\sigma \cong \Pi^X(\fH)$ to be $\dim_P(\fH) = \dim(\fH_{\comp})$. It is one of the main results of \cite{Brothier-Wijesena23} that this is well-defined and thus is an invariant for the representation $\sigma$.

\subsubsection{Diffuse and atomic P-representations} \label{subsec:diff-atom-rep}
Given $\sigma^X := \Pi^X(\fH)$ we have $$\sigma^X \cong \sigma^X_{\diff} \oplus \sigma^X_{\atom}$$
where $\sigma^X_{\diff} := \Pi^X(\fH_{\diff})$ and $\sigma^X_{\atom} := \Pi^X(\fH_{\atom})$. 
We refer to $\sigma^X_{\diff}$ (resp.~$\sigma^X_{\atom}$) as the diffuse (resp.~atomic) part of $\sigma^X$.

\subsection{Tensor categories} \label{subsec:intro-tensor-cat}

We recall some definitions from category theory. 
We follow the terminology of the book \cite{Etingof-Gelaki-Nikshych-Ostrik16} that we refer for details.

\begin{definition}
	Let $\cC$ be a symmetric monoidal category with the tensor product bi-functor $\boxtimes:\cC\times\cC\to\cC$ and unit object $(\bone,\iota).$
	It is
	\begin{itemize}
	\item a \emph{dagger category} if for all objects $X,Y\in \ob(\cC)$ there exists an anti-linear map $\dagger : \Hom_{\cC}(X,Y) \rightarrow \Hom_{\cC}(Y,X)$ satisfying $f^{\dagger\dagger} = f$, $(f \circ g)^\dagger = g^\dagger \circ f^\dagger$, and $\id_Z^\dagger = \id_Z$ for all $Z\in \ob(\cC)$;
	\item \emph{unitary} (also referred to as $C^*$-tensor category) if it is a dagger category and the endomorphism spaces are $C^*$-algebras.
	\item \emph{rigid} if every object $X$ has a dual $X^*$: there exist an evaluation map $\textrm{ev}_X : X^* \boxtimes X \rightarrow \bone$ and a co-evaluation map $\textrm{coev}_X : \bone \rightarrow X \boxtimes X^*$ satisfying		
	\begin{align*}
		(\id_X \boxtimes \textrm{ev}_X) \circ a_{X,X^*,X} \circ (\textrm{coev}_X \boxtimes \id_X) &= \id_X \\
		(\textrm{ev}_X \boxtimes \id_{X^*}) \circ a^{-1}_{X^*,X,X^*} \circ (\id_{X^*} \boxtimes \textrm{coev}_{X}) &= \id_{X^*}
	\end{align*}
	where $a_{Z_1, Z_2, Z_3} : (Z_1 \boxtimes Z_2) \boxtimes Z_3 \rightarrow Z_1 \boxtimes (Z_2 \boxtimes Z_3)$ is the associativity constraint.		
	\end{itemize}
	
	A \emph{tensor category} $(\cC, \boxtimes)$ over the complex numbers $\C$ is a $\C$-linear (i.e.~morphism spaces are $\C$-vector spaces and compositions are bi-linear) rigid symmetric monoidal category such that $\End_{\cC}(\bone) \cong \C$. It is \emph{pivotal} if the mapping $X \mapsto X^{**}$ induces a monoidal natural isomorphism.
\end{definition}

Pivotal tensor categories admit \emph{traces} and \emph{quantum dimensions}.

\begin{definition}
	Let $\cC$ be a pivotal tensor category. Then for an object $X$ of $\cC$ the \textit{categorical trace} is defined as:
	\[\textrm{tr} : \End_{\cC}(X) \rightarrow \End_{\cC}(\bone),\ f \mapsto \textrm{ev}_X \circ (\id_{X^*} \circ f) \circ \gamma_X \circ \textrm{coev}_X\]
	where $\End_{\cC}(\bone)$ is identified with $\C$ and $\gamma_X$ is the isomorphism between $X \boxtimes X^*$ and $X^* \boxtimes X$. 
	The \textit{quantum dimension} of $X$ is then defined as $\textrm{tr}(\id_X)$.
\end{definition}

Here is one example that we will later use. 
\begin{example}\label{ex:Hilb}
	Let $\Hilb$ be the category of complex Hilbert spaces with bounded linear maps for morphisms.
	Equip $\Hilb$ with the usual monoidal product $\ot$ which confers to $\Hilb$ a structure of a unitary symmetric $\C$-linear monoidal category with the dagger structure being the usual adjoint of operators and the unit object is $\C$.
	The associativity constraints of $(\Hilb,\ot)$ are the unique continuous linear maps
	$$\alpha_{\fH,\fK,\fL}: (\fH\ot \fK)\ot \fL\to \fH \ot (\fK\ot \fL)$$
	defined on elementary tensors by
	$$(\xi\ot \eta)\ot \zeta\mapsto \xi\ot (\eta\ot \zeta).$$
\end{example}

\section{A Tensor Product of Pythagorean Representations}\label{sec:tensor-product}
We now define a binary operation $\boxtimes$ on certain pairs of P-modules. 
We will then restrict it to certain sub-classes of P-modules that are closed under $\boxtimes$ and on which the P-functors are fully faithful.
From there we will transport $\boxtimes$ from P-modules to P-representations.
Recall all P-modules are assumed to be finite-dimensional.

\begin{notation} \label{not:polar-decomp}
	Let $T$ be a bounded linear operator on a \emph{finite-dimensional} Hilbert space $\fH$. Write $T = U_T\vert T \vert$ to be \emph{any} polar decomposition where $U_T$ is unitary and $\vert T \vert := (T^*T)^{1/2}$.
\end{notation}

Note, the choice of $U_T$ is unique if and only if $T$ is invertible. For our purpose, the choice of $U_T$ does not matter and we will let $U_T$ be any suitable unitary operator.

\subsection{A binary operation on P-modules.}
Consider two P-modules $(A,B,\fH)$, $(\ti A, \ti B, \ti \fH)$. 
A na\"ive choice for defining a tensor product would be to take:
\[(A\otimes \ti A, B \otimes \ti B, \fH \otimes \ti \fH).\]
This attempt fails: the pair of operators does not satisfy the Pythagorean equality in general. Nevertheless, this attempt can be partially salvaged as below.

\begin{definition} \label{def:tensor-product}
Let $(A,B,\fH)$, $(\ti A, \ti B, \ti \fH)$ be any two P-modules satisfying
\begin{equation}\label{eq:condition-P}\ker(A \otimes \ti A) \cap \ker(B \otimes \ti B) = \{0\}.\end{equation}	
Then define a binary operation $\boxtimes$ by:
	\[(A,B,\fH) \boxtimes (\ti A, \ti B, \ti \fH) = ((A \otimes \ti A)K^{-1}, (B \otimes \ti B)K^{-1}, \fH \otimes \ti \fH)\]
	where 
	\[K := K(A,\ti A) = \sqrt{\vert A \vert^2 \otimes \vert \ti A \vert^2 + \vert B \vert^2 \otimes \vert \ti B \vert^2}.\]
	The above resulting triple is a P-module and we call $\boxtimes$ the \textit{tensor product} of P-modules. 
	
	Note, $K$ purely depends on $A$ and $\ti A$ since $\vert B\vert^2 = \id_{\fH} - \vert A \vert^2$ and $\vert \ti B \vert^2 = \id_{\ti \fH} - \vert \ti A \vert^2$.
\end{definition}

\textbf{Terminology and notations.}
We use the term ``tensor product'' and the symbol $\boxtimes$ to refer to our novel binary operation while ``usual tensor product" and $\ot$ refer to the classical one appearing in Example \ref{ex:Hilb}.

\begin{remark} \label{rem:tensor-product-def}
\begin{enumerate}
\item The operator $K$ of above is indeed invertible.
Since $\ker(T)=\ker(T^*T)$ and $|A\ot \ti A|=|A|\ot |\ti A|$ we deduce that the assumption \eqref{eq:condition-P} is equivalent to $\ker(K)=\{0\}$. Since $\dim(\fH\ot\ti\fH)<\infty$ this is equivalent to $K$ being invertible. 
\item The Pythagorean equality is generically satisfied only when $K^{-1}$ is placed to the right of $A\ot \ti A$ and $B\ot \ti B$.
This limits our choices for defining $\boxtimes$.
\item 
The above provides a general process to construct a P-module from any pair of operators $X,Y \in B(\fH)$ satisfying $\ker(X) \cap \ker(Y) = \{0\}$. Indeed,
\[(XK^{-1}, YK^{-1}, \fH)\]
forms a P-module where $K = \sqrt{\vert X \vert^2 + \vert Y \vert^2}$.
\item Observe that
\begin{align*}
(A\otimes \ti A)K^{-1} &= (U_A \otimes U_{\ti A})(\vert A \vert \otimes \vert \ti A \vert)K^{-1} \\
&= U_A \otimes U_{\ti A}\ \frac{\vert A \vert \otimes \vert \ti A \vert}{\sqrt{\vert A \vert^2 \otimes \vert \ti A \vert^2 + \vert B \vert^2 \otimes \vert \ti B \vert^2}}\\
&= U_A \otimes U_{\ti A}\ \sqrt{\frac{\vert A \vert^2 \otimes \vert \ti A \vert^2}{\vert A \vert^2 \otimes \vert \ti A \vert^2 + \vert B \vert^2 \otimes \vert \ti B \vert^2}}
\end{align*}
where the above fraction notation is justified because the two terms commute (hence there is no ambiguity in the order of composition).
This gives a practical polar decomposition for $(A \otimes \ti A)K^{-1}$. 
\end{enumerate}	
\end{remark}

Motivated by the above remark we introduce the following binary operation.
For an operator $P\in B(\fH)$ we write $0\leq P\leq \id$ when $0\leq \langle P\xi,\xi\rangle\leq \langle \xi,\xi\rangle$ for all $\xi\in\fH$, i.e.~$P$ is positive and dominated by $\id$.

\begin{definition} \label{def:star-operation}
Given Hilbert spaces $\fH,\ti\fH$ define 
$$C(\fH,\ti\fH):=\{(P,Q)\in B(\fH)\times B(\ti \fH):\ 0\leq P, Q \leq \id \text{ satisfying } \eqref{eq:condition-PQ} \}$$
where
	\begin{equation}\label{eq:condition-PQ}\ker(P \otimes Q) \cap \ker((\id - P) \otimes (\id - Q)) = \{0\}.\end{equation}
Then define 
\begin{align*}
& C(\fH,\ti\fH)\to B(\fH\ot\ti\fH)\\
& (P,Q)\mapsto P\star Q:=\frac{P \otimes Q}{\sqrt{P^2 \otimes Q^2 + (\id - P^2) \otimes (\id - Q^2)}}.
\end{align*}
Note that $0\leq P\star Q\leq \id$. 
In the scalar case $P=p,Q=q\in [0,1]$ and
$$p\star q=\frac{pq}{\sqrt{p^2q^2 + (1-p^2)(1-q^2)}}$$
which is defined on the unit square $[0,1]^2$ minus the two points $(0,1), (1,0)$.
\end{definition}

\begin{remark} \label{rem:star-operations}
Note, for two positive operators $U\in B(\fH), V \in B(\ti\fH)$ dominated by the identity, we have that $W:=U^2 \otimes V^2 + (\id - U^2) \otimes (\id - V^2)$ is invertible if and only if $(U,V) \in C(\fH, \ti\fH)$ (this can be verified by computing the spectrum of $W$).
	\end{remark}

Recall that all positive operators acting on a finite-dimensional Hilbert space admits an orthonormal basis of eigenvectors. Moreover, we have the following fact.

\begin{lemma} \label{lem:star-diag}
	Let $(P,Q)$ be a pair of positive operators satisfying \eqref{eq:condition-PQ}.
	Let $\{\xi_i\}_i$ and $\{\eta_j\}_j$ be some orthonormal basis of eigenvectors of $P$, $Q$, respectively, with associated eigenvalues $\{\lambda_i\}_i$ and $\{\mu_j\}_j$, respectively. Then $P \star Q$ admits an orthonormal basis of eigenvectors which is $\{\xi_i \otimes \eta_j\}_{i,j}$ with associated eigenvalues $\{\lambda_i \star \mu_j\}_{i,j}$.
\end{lemma}

We now show that $\star$ is associative up to the associativity constraints of $(\Hilb,\ot)$.

\begin{lemma} \label{lem:star-associative}
Consider $P,Q,R$ assuming that $P\star Q$ and $(P\star Q)\star R$ are well-defined.
Then $Q\star R$ and $P\star (Q\star R)$ are well-defined, the associativity constraints of $(\Hilb,\ot)$ conjugates $(P\star Q)\star R$ with $P\star (Q\star R)$, and we have 
$$(P\star Q)\star R = (P\ot Q)\ot R \cdot ( (P^2\ot Q^2)\ot R^2 + ((\id-P^2)\ot (\id-Q^2))\ot (\id-R^2))^{-1/2}.$$
\end{lemma}

\begin{proof}
Let $P,Q,R$ be as above.
Write $P_0$ for $\sqrt{\id-P^2}$ and recall $K(P,Q)$ from Definition \ref{def:tensor-product}. 
Observe that $K(P_0,Q_0)=K(P,Q).$
Moreover,
\begin{align}\label{eq:star}
(P\star Q)_0^2  & =  \id - (P\star Q)^2 = \id - \left(\frac{P\ot Q}{\sqrt{ P^2\otimes Q^2 + P_0^2\otimes Q_0^2}}\right)^2
 = \frac{P_0^2\ot Q_0^2}{P^2\otimes Q^2 + P_0^2\otimes Q_0^2}\\
& = (P_0\star Q_0)^2.\nonumber
\end{align}
This implies that $(P\star Q)\star R$ is equal to:
\begin{align*}
&= \frac{(P \star Q) \otimes R}{\sqrt{(P \star Q)^2 \otimes R^2 + (P\star Q)_0^2 \otimes R_0^2}} \\
& = \frac{(P \star Q) \otimes R}{\sqrt{(P \star Q)^2 \otimes R^2 + (P_0\star Q_0)^2 \otimes R_0^2}}  \text{ by Equation \ref{eq:star} }\\
& = \frac{(P \ot Q) \otimes R \circ (K(P,Q)^{-1}\ot \id)}{\sqrt{((P^2 \ot Q^2) \otimes R^2)\circ (K(P,Q)^{-2}\ot \id) + ((P_0^2\ot Q_0^2) \otimes R_0^2)  \circ (K(P_0,Q_0)^{-2}\ot \id)}} \\
& = \frac{(P \ot Q) \otimes R}{\sqrt{(P^2 \ot Q^2) \otimes R^2 + (P_0^2\ot Q_0^2)\ot R_0^2 }}. 
\end{align*}
Conjugating this operator by the associativity constraint of $(\Hilb,\ot)$ (see Remark \ref{rem:star-operations}) shows that $P\star (Q\star R)$ is well-defined and equal to the above formula, up to moving the parenthesis.
\end{proof}

The $\star$ operation of above immediately provides a convenient alternate formulation for $\boxtimes$. 

\begin{proposition} \label{prop:alt-tensor-product-def}
	Let $(A,B,\fH)$, $(\ti A, \ti B, \ti \fH)$ be two P-modules satisfying \eqref{eq:condition-P}. Then:
	\[(A,B,\fH) \boxtimes (\ti A, \ti B, \ti \fH) = ((U_A \otimes U_{\ti A}) \circ (\vert A \vert \star \vert \ti A \vert), (U_B \otimes U_{\ti B}) \circ (\vert B \vert \star \vert \ti B \vert), \fH \otimes \ti \fH).\]
\end{proposition}

\begin{remark} \label{rem:alt-tensor-product-def}
The above formulation does not depend on the choices of polar decompositions of $A$ and $B$. Moreover, this formulation along with Lemma \ref{lem:star-diag} demonstrates that $\boxtimes$ is nothing more than the usual tensor product of the individual operators of the P-modules, but now the positive part $\vert A\vert \otimes \vert\ti A \vert$ (resp. $\vert B\vert \otimes \vert\ti B\vert$) has been replaced by $\vert A \vert \star \vert \ti A \vert$ (resp.~$\vert B \vert \star \vert \ti B\vert$) which can be considered as a ``weighted quadratic mean'' of the individual positive parts to ensure that the Pythagorean equality is satisfied.
\end{remark}

Below we enunciate some useful properties of the tensor product $\boxtimes$. 

\begin{lemma} \label{lem:tensor-prod-prop}
	Let $m, \ti m$ be two P-modules. The operation $\boxtimes$ satisfies the following properties (whenever $\boxtimes$ is defined):
	\begin{enumerate}
		\item $\boxtimes$ is associative up to conjugating by the associativity constraints of $(\Hilb,\ot)$;
		\item $\boxtimes$ is distributive with respect to the usual direct sum; 
		\item $\boxtimes$ is well-defined on equivalence classes of P-modules, i.e.~if $m\cong m_0, \ti m\cong\ti m_0$, then $m\boxtimes \ti m\cong m_0\boxtimes \ti m_0$;	
		\item $m\boxtimes\ti m \cong \ti m\boxtimes m$;
		\item if $\bone := (1/\sqrt{2},1/\sqrt{2},\C)$, then $m\boxtimes \bone\cong \bone\boxtimes m\cong m$. 	
	\end{enumerate}
\end{lemma}

\begin{proof}	
	\textit{Proof of (1).} This follows from Lemma \ref{lem:star-associative} and Proposition \ref{prop:alt-tensor-product-def}.
	
	\textit{Proof of (2).} It is a standard fact that the polar decomposition is compatible with the direct sum in the sense that we can choose $U_{T \oplus S}$ to be equal to $U_T \oplus U_S$. Since the usual tensor product is distributive with respect to the direct sum, again by Proposition \ref{prop:alt-tensor-product-def}, it is suffice to show that $\star$ is distributive with respect to the direct sum whenever it is defined. 
	This is rather obvious. 
	Indeed, for any three positive operators $P_1, P_2, P_3$, the natural unitary transformation 
	$$(\xi_1, \xi_2) \otimes \xi_3\mapsto(\xi_1 \otimes \xi_3, \xi_2 \otimes \xi_3)$$
	conjugates the operators $(P_1\oplus P_2)\star P_3$ and $(P_1\star P_3)\oplus (P_2\star P_3)$.
	This shows that $\star$ is right-distributive. 
	A similar proof yields left-distributivity.

	\textit{Proof of (3).} 
	Consider some P-modules $m,\ti m,m_0,\ti m_0$ so that $m,m_0$ and $\ti m,\ti m_0$ are equivalent with intertwinners $u$ and $\ti u$, respectively. If a unitary intertwinnes two operators $A$ and $\ti A$, then it also intertwines $\vert A \vert$ and $\vert \ti A \vert$. Thus, by Definition \ref{def:tensor-product} it follows that $u \otimes \ti u$ intertwinnes $m \boxtimes \ti m$ and $m_0 \boxtimes \ti m_0$.
	
	\textit{Proof of (4).} Given some Hilbert spaces $\fH,\ti\fH$, define the flip operator
	$$u:\fH\ot\ti\fH\to \ti\fH\ot \fH,\ \xi\ot \ti \xi\mapsto \ti\xi\ot \xi$$
	(the unique linear map satisfying the formula of above).
	If $m=(A,B,\fH),\ti m=(\ti A,\ti B,\ti\fH)$ are P-modules, then conjugating the P-module $m\boxtimes \ti m$ by $u$ yields $\ti m\boxtimes m$. 
	
	\textit{Proof of (5).} This follows from the definition and obvious computations.
\end{proof}

\subsection{Monoidal categories of P-modules}
In this section we define two classes of P-modules closed under $\boxtimes$ which in turn define two well-behave monoidal categories.

\begin{definition}\label{def:N-M}
	Let $\NI$ (resp.~$\NII$) be the class of finite-dimensional P-modules $(A,B,\fH)$ where $A$ is normal (resp.~$A$ is normal and invertible) and $B$ is invertible.
\end{definition}

\begin{proposition} \label{prop:normal-full-mod}
Any P-module of $\NI$ is full and any P-module of $\NII$ is diffuse.
\end{proposition}

\begin{proof}
We will use the following characterisation of fullness:
a P-module $(A,B,\fH)$ is full when for every orthogonal projection $p$ satisfying that $Ap=pAp$ and $Bp=pBp$ we have that $pA=Ap$ and $pB=Bp$.
Indeed, this means that if $\fK\subset \fH$ is a P-submodule, then $\fK^\perp$ is also a P-submodule. This prevents $\fH$ to contain a non-trivial residual subspace.

Fix $(A,B,\fH)$ in $\NI$.
Let $p$ be an orthogonal projection satisfying that $Ap=pAp$ and $Bp=pBp$.
Since $A$ is normal, it is a classical fact that $Ap=pAp$ implies $Ap=pA$.
Since $p$ is self-adjoint its relative commutant is a *-subalgebra of $B(\fH)$ and in particular is closed under continuous functional calculus.
We deduce that $|A|:=\sqrt{A^*A}, |B|=\sqrt{\id-A^*A}$, and $|B|^{-1}$ commute with $p$.
Subsequently, $$B|B|^{-1} p = B p |B|^{-1} = pBp|B|^{-1} = pB |B|^{-1} p.$$
Since $B|B|^{-1}$ is normal we obtain that $pB|B|^{-1}=B|B|^{-1}p$ and thus $pB=Bp$.

Consider $(A,B,\fH)\in \NII$.
Since $A,B$ are invertible so are $|A|$ and $|B|$.
The Pythagorean equality $|A|^2+|B|^2=\id$ and the fact that $\dim(\fH)$ is finite imply that the spectra of $|A|$ and $|B|$ are strictly contained between $0$ and $1$.
In particular, the norms $\|A\|,\|B\|$ are both strictly smaller than $1$. 
This immediately implies that all infinite increasing sequences of words in $A,B$ will converge to $0$ and thus $(A,B,\fH)$ is diffuse.
\end{proof}

\begin{proposition} \label{prop:P-rep-tensor-prod}
	The following assertions are true:
	\begin{enumerate}
		\item $\NI$ and $\NII$ are closed under taking P-submodules, finite direct summands, and $\boxtimes$;
		\item the operation $\boxtimes$ restricted to $\NI$ is multiplicative with respect to the P-dimension: 
		$$\dim_P(m \boxtimes \ti m) = \dim_P(m) \cdot \dim_P(\ti m) \text{ for all } m,\ti m\in\NI.$$
	\end{enumerate}
\end{proposition}

\begin{proof}
	It is obvious that $\NI,\NII$ are closed under taking P-submodules and finite direct summands. Let $m = (A,B,\fH)$, $\ti m = (\ti A, \ti B, \ti \fH) \in \NI$ and $K := K(A,B)$ (see Definition \ref{def:tensor-product}). Firstly, since $B,\ti B$ are invertible, the tensor product $m \boxtimes m$ is defined. As well, it is clear that $(B \otimes \ti B)K^{-1}$ is also invertible. 
By assumption $A$ and $\ti A$ are normal. Hence, $A$ commutes with $A^*$, thus with $|A|$, and also with $|B|$ using the Pythagorean equality.
We deduce that $A\ot \ti A$ commutes with $K^{-1}$.
Now, $((A\ot \ti A)K^{-1})^* = K^{-1} (A^*\ot \ti A^*)$ which commutes with $A\ot \ti A$ and thus with $(A\ot \ti A)K^{-1}$.	
This proves that $m\boxtimes \ti m$ is in $\NI$.

Further assume now that $A,\ti A$ are invertible. A similar argument than above implies that $(A\ot \ti A)K^{-1}$ is invertible.
We have proven that $\NI$ and $\NII$ are closed under $\boxtimes$.

Item (2) comes from the fact that if $m=(A,B,\fH)$ is full, then $\dim_P(m)=\dim(\fH)$.
We may then conclude using Proposition \ref{prop:normal-full-mod}.
\end{proof}

We identify the collections $\NI,\NII$ with full sub-categories of $\Mod(\cP)$. We can now define a monoidal structure on them.

\begin{theorem} \label{thm:tensor-category}
	The operations 
	\[(m, \ti m) \mapsto m \boxtimes \ti m,\ (\theta,\ti \theta) \mapsto \theta \otimes \ti \theta\]
	where $m, \ti m$ are P-modules and $\theta, \ti \theta$ morphisms defines a structure of $\C$-linear symmetric monoidal categories on $\NI$ and $\NII$ that we continue to denote by $\boxtimes$. The associativity constraints are the restrictions of those of $(\Hilb,\ot).$
Moreover,
	\begin{enumerate}
		\item $(\NII, \boxtimes)$ is a rigid, pivotal, unitary tensor category;
		\item the \emph{quantum dimension} of $(\NII,\boxtimes)$ coincides with the P-dimension.
	\end{enumerate}
\end{theorem}

\begin{proof}
Let us show that $\boxtimes$ indeed defines a bi-functor on $\NI$. 
From the above it remains to show that the tensor product of morphisms is also a morphism. Let $m_i = (A_i, B_i, \fH_i)$, $\ti m_i = (\ti A_i, \ti B_i, \ti \fH_i)$ be P-modules in $\NI$ and let $\theta_i:m_i\to \ti m_i$ be a morphism between for $i=1,2$. Recall that we have
	\[m_1 \boxtimes m_2 = ((A_1\otimes A_2)K^{-1}, (B_1 \otimes B_2)K^{-1}, \fH_1\ot\fH_2)\]
	where 
	\[K = \sqrt{A_1^*A_1 \otimes A_2^*A_2 + B_1^*B_1\otimes B_2^*B_2}\]
	and a similar expression exists for $\ti m_1 \boxtimes \ti m_2$.
	By definition $\theta_i$ intertwinnes $A_i$ with $\ti A_i$ and intertwinnes $B_i$ with $\ti B_i$. 
	Since the P-modules are full and finite-dimensional we deduce that the $\theta_i$ intertwinnes the adjoint as well by using \cite[Theorem 4.13]{Brothier-Wijesena23} item (1) applied to $Z=\cO$.
	Hence, $K$ and $\ti K$ are intertwinned by $\theta_1\ot\theta_2$ implying that $\theta_1\ot\theta_2$ intertwinnes the P-module $m_1\boxtimes m_2$ with $\ti m_1\boxtimes \ti m_2.$
	
We have previously observed that the usual associativity constraints of $(\Hilb,\ot)$ provide associativity constraints for $\boxtimes.$ Hence, they automatically fulfill the pentagon axiom since they do on $\Hilb.$
Considering $\bone=(1/\sqrt 2,1/\sqrt 2,\C)$ we easily obtain the unit axiom.
All together this proves that $(\scrM,\boxtimes)$ is a monoidal category.
Conjugating by the flip unitary transformation for two Hilbert spaces provides the symmetric structure.
	
	We now prove rigidity of $\NII$.
	Let $m = (A,B,\fH)$ be a P-module in $\NII$. 
	Denote by $\fH^*$ the dual space $B(\fH, \C)$ of $\fH$. There exists a canonical anti-linear isometry $\fH \ni \xi \mapsto \xi^* \in \fH^*$ 
	defined as $\xi^*:\fH\to\C, \eta\mapsto\langle \eta, \xi \rangle$.
	The dual space $\fH^*$ is equipped with an inner-product given by
	\[\langle \xi^*, \eta^* \rangle = \langle \eta, \xi \rangle\ \textrm{for all } \xi, \eta \in \fH\]
	which makes $\fH^*$ a Hilbert space. 
	For an operator $T \in B(\fH)$ define $\widehat T \in B(\fH^*)$ by
	\[\widehat T(\xi^*)(\eta) := \xi^*(T^*\eta) = \langle T^*\eta, \xi \rangle = \langle \eta, T\xi\rangle \text{ for } \xi,\eta\in\fH.\]
	Thus, $\widehat T(\xi^*) = (T\xi)^*$. 
		
	Since $m=(A,B,\fH)$ is in $\NII$ we have by definition that $A$ and $B$ are invertible. 
	Therefore, they admit a {\it unique} polar decomposition $A=U_A|A|$ and $B=U_B|B|$.		
	We claim that the dual of $m$ exists and is given by
	\[m^* := (\widehat U_A \vert \widehat B \vert, \widehat U_B \vert \widehat A \vert, \fH^*).\]
	The evaluation and co-evaluation maps are the one of finite-dimensional Hilbert spaces:
	\begin{align*}
		\textrm{ev}_m : \fH^*\ot \fH\to \C, & \ \xi^*\ot \eta \mapsto \xi(\eta) = \langle \eta,\xi\rangle; \\
		\textrm{coev}_m : \C\to \fH\ot \fH^*, & \ 1 \mapsto \sum_{i\in I} \xi_i \otimes \xi_i^*,
	\end{align*}
	where $\{\xi_i\}_{i\in I}$ is \emph{any} orthonormal basis of $\fH$ for some index set $I$ (noting that $\{\xi_i^*\}_{i\in I}$ defines an orthonormal basis of $\fH$).
	Next, observe for general operators $T,S \in B(\fH)$ we have $\widehat T \widehat S = \widehat{TS}$ and moreover $(\widehat T)^* = \widehat{T^*}$.
	Applying the above to $m^*$ we obtain
	\[(\widehat U_A \vert \widehat B \vert)^*(\widehat U_A \vert \widehat B \vert) + (\widehat U_B \vert \widehat A \vert)^*(\widehat U_B \vert \widehat A \vert) = \id_{\fH^*}\]
	and thus $m^*$ is a P-module.
	
	The maps $\textrm{ev}_m$, $\textrm{coev}_m$ are the standard evaluation and co-evaluation maps in the rigid category of finite-dimensional Hilbert spaces.
	Hence, the above maps satisfy the so-called zig-zag relations. It remains to show that $\textrm{ev}_m$ is an intertwinner from $m^* \boxtimes m$ to $\bone$ and $\textrm{coev}_m$ is an intertwinner from $\bone$ to $m \boxtimes m^*$.
	
	Using Proposition \ref{prop:alt-tensor-product-def} we have
	\[m^* \boxtimes m = ((\widehat U_A \otimes U_A) (\vert \widehat B \vert \star \vert A \vert), (\widehat U_B \otimes U_B) (\vert \widehat A \vert \star \vert B \vert), \fH^* \otimes \fH)\]
	and a similar expression for $m \boxtimes m^*$. 
	Fix an orthonormal basis $\{\xi_i\}_i$ of $\fH$ such that each $\xi_i$ is a common eigenvector of $\vert A \vert$ and $\vert B \vert$ with strictly positive eigenvalues $a_i$ and $b_i$, respectively. 
	Note that such a basis exists since $|A|$ and $|B|$ are invertible finite-dimensional positive operators that mutually commute. By the Pythagorean relation we have $a_i^2 + b_i^2 = 1$.
	Observe that	
	\[\vert \widehat A \vert (\xi_i)^* = (\vert A \vert \xi_i)^* = (a_i\xi_i)^* = a_i(\xi_i)^*\]
	since $a_i$ is a positive scalar. Similarly, we have $\vert \widehat B \vert (\xi_i)^* = b_i(\xi_i)^*$. Subsequently, $\{\xi_i^* \otimes \xi_j\}_{i,j}$ forms an orthonormal basis of eigenvectors of $\vert \widehat B \vert \otimes \vert A \vert$ and $\vert \widehat A \vert \otimes \vert B \vert$ with eigenvalues $\{b_ia_j\}_{i,j}$ and $\{a_ib_j\}_{i,j}$, respectively.
	Then Lemma \ref{lem:star-diag} shows that
	\[(\vert \widehat B \vert \star \vert A \vert)(\xi_i^* \otimes \xi_i) = (b_i \star a_j)\xi_i^* \otimes \xi_j = (\sqrt{1-a_i^2}\star a_j)\xi_i^* \otimes \xi_j.\]
	Note that $\sqrt{1-a_i^2}\star a_j = 1/2$ when $i = j$. Thus we have:
	\begin{align*}
		\textrm{ev}_m((\widehat U_A \otimes U_A) (\vert \widehat B \vert \star \vert A \vert)(\xi_i^*\otimes \xi_j)) &= \sqrt{(1-a_i^2)\star a_j^2} \cdot \textrm{ev}_m(\widehat U_A \otimes U_A(\xi_i^*\otimes \xi_j)) \\
		&= \sqrt{(1-a_i^2)\star a_j^2} \cdot \textrm{ev}_m((U_A\xi_i)^* \otimes U_A(\xi_j))\\
		&= \sqrt{(1-a_i^2)\star a_j^2} \cdot \langle U_A\xi_j, U_A\xi_i \rangle\\
		&= \frac{1}{\sqrt{2}} \cdot \langle \xi_j, \xi_i \rangle = \frac{1}{\sqrt{2}} \cdot \textrm{ev}_m(\xi_i^* \otimes \xi_j).
	\end{align*}
	We also have a similar calculation for $(\widehat U_B \otimes U_B) (\vert \widehat A \vert \star \vert B \vert)$. This shows that $\textrm{ev}_m$ is an intertwinner between $m^* \boxtimes m$ and $\bone = (1/\sqrt{2}, 1/\sqrt{2}, \C)$.
A similar computation provides that $\textrm{coev}_m$ is an intertwinner. 
This proves rigidity of $\NII$.
	
	The category $(\NII,\boxtimes)$ inherits the pivotal structure of finite-dimensional Hilbert spaces.
	For unitarity, we use the fact that the relative commutant of a normal operator is closed under adjoint. This implies that taking adjoints defines a dagger structure on $\NII$ and that the endomorphism spaces are *-closed. It is then a von Neumann algebra (as being a relative commutant and being closed under the adjoint) and thus is a $C^*$-algebra.

	Finally, since $\textrm{ev}_X$, $\textrm{coev}_X$ are the usual evaluation and co-evaluation maps of the category of finite-dimensional Hilbert spaces, it is immediate that the quantum dimension of a P-module coincides with the usual dimension of the underlying Hilbert space of the P-module. In particular, since P-modules in $\NII$ are full, it follows that the quantum dimension is equal to the Pythagorean dimension.
\end{proof}

\begin{remark} \label{rem:tensor-cat}
\begin{enumerate}
\item The tensor categories $\NI$ and $\NII$ contain continuously many objects that are pair-wise non-isomorphic (this follows easily from \cite{Brothier-Wijesena23}) and thus are not finitely generated nor countably generated.
\item Let $(e_1,\dots,e_d)$ be the usual co-ordinate basis of $\C^d$, write $\fH:=\C^d$, and consider a P-module $m=(A,B,\C^d)=(U_A|A|,U_B|B|,\C^d)$ made of matrices. We know an anti-isometry $\xi\mapsto \xi^*$ from $\fH$ to its dual. However, using $(e_1,\dots,e_d)$ we can construct a \emph{linear} isometry from $\fH$ to $\fH^*$ as follows:
$$\kappa:\fH\to\fH^*,\ \sum_{j=1}^d \lambda_j e_j\mapsto \sum_{j=1}^d \lambda_j e_j^*.$$
Under the identification given by $\kappa$ we deduce the formula 
$$m^*=(U_A^\dag |B|^\dag, U_B^\dag |A|^\dag, \C^d)$$
where $Z^\dag$ is the entry-wise complex conjugate of a matrix $Z$.
\item Note in the proof that $(\NII, \boxtimes)$ is a tensor category, the only properties of $\NII$ used was that it is closed under $\boxtimes$, its P-modules are full, and both operators for a P-module in $\NII$ are invertible. 
\item The definitions of $\NI$ and $\NII$ are asymmetric and we have not find any suitable symmetric collections that contain them.
However, we may consider the sub-class of $\NII$ where additionally $B$ is normal. 
This latter class forms a monoidal category that can be shown to be isomorphic to $(\Rep^{\FD}(\F_2 \times \R), \otimes)$: the tensor category of finite-dimensional (unitary) representations of the free group of rank two direct product with $\R$
\end{enumerate}
\end{remark}

\subsection{A binary operation on P-representations.} \label{subsec:tensor-product-rep}

\begin{notation}
We use the super-script $\FD$ and sub-script $\full$ to write the various full sub-categories of $\Mod(\cP), \Rep(X)$ where $\FD$ stands for finite P-dimensional and $\full$ for full. Moreover, write $\Mod^{\FD+}(\cP)$ for the full sub-category consisting of finite P-dimensional P-modules which do not contain any one-dimensional atomic P-modules. Similarly define $\Rep^{\FD+}(X)$.
\end{notation}

To transport $\boxtimes$ from P-modules to P-representations we rely on the following classification which is one of the main results in \cite{Brothier-Wijesena23}.

\begin{theorem} \label{theo:classification}
Let $m=(A,B,\fH)$ be a finite P-dimensional P-module and consider $\Pi^\cO(\fH)$.
We have that $M:=(\sigma(s_0^*), \sigma(s_1^*),  \scrH)$ is a P-module (of infinite dimension) that admits  a \emph{smallest} complete P-submodule $\fH_{\textrm{small}}$. 
After identifying $m$ with a P-submodule of $M$ we obtain that $\fH_{\textrm{small}}$ is equal to the complete submodule of $m$.
	Moreover, the following assertions are true. Take $Z = T,V,\cO$.
	\begin{enumerate}
			\item The functor $\Pi^Z$ restricts to a fully faithful functor from $\Mod^{\FD}_{\full}(\cP)$ to $\Rep^{\FD}(Z)$.
			\item There exists a fully faithful functor $$M^Z : \Rep^{\FD}(Z) \rightarrow \Mod^{\FD}_{\full}(\cP).$$
			\item The pair $(\Pi^Z, M^Z)$ defines an equivalence of categories between $\Mod^{\FD}_{\full}(\cP)$ and $\Rep^{\FD}(Z)$.
	\end{enumerate}
	The above three statements hold true for $F$ after replacing the super-script $\FD$ with $\FD+$.
\end{theorem}

This implies that if $m, \ti m$ are finite P-dimensional P-modules, then $\Pi^Z(m)$ is irreducible if and only if the complete P-submodule of $m$ is irreducible, and $\Pi^Z(m) \cong \Pi^Z(\ti m)$ if and only if the complete P-submodules of $m, \ti m$ are equivalent for $Z=T,V,\cO$. This is true for $F$ if $m, \ti m$ additionally do not contain any one-dimensional atomic P-modules.

We will now briefly explain how to construct the functor $M^X$ as this will be integral for transporting $\boxtimes$ to P-representations.
It consists in first extending a representation of $X$ to $\cO$ and then extracting the smallest complete P-submodule.
In details, consider first $Y=T,V$ (the $F$-case being slightly more subtle).
Fix $(\sigma^Y, \scrH) \in \Rep^{\FD}(Y)$.
If $\sigma^Y$ is diffuse then via a canonical limit process, two isometries $S_0,S_1 \in B(\scrH)$ can be constructed satisfying $S_0S_0^*+S_1S_1^* = \id$. This construction relies on choosing any $g \in Y$ such that $\supp(g) := \overline{\{p \in \cC : g(p) \neq p\}} = 1 \cdot \cC$ (see the proof of Proposition 3.1 in \cite{Brothier-Wijesena23} for more details). This yields a representation $\sigma^\cO$ of $\cO$ and it is immediate by construction that $\sigma^\cO$ is the \textit{unique} extension of $\sigma^Y$. 
If $\sigma^Y$ is atomic (say $\sigma^Y$ contains irreducible atomic representations of P-dimension $d_1, \dots, d_n$) then the above construction can be easily generalised (now further require the chosen $g \in Y$ to not fix any rays with a period length in $d_1, \dots, d_n$) to uniquely extend $\sigma^Y$ to $\sigma^\cO \act \cO$. 
Then from $(\sigma^\cO, \scrH)$, by the above theorem we can identify the smallest complete P-submodule $\fH_{\textrm{small}} \subset \scrH$.  

From the above process, it can be observed if $\theta$ is an intertwinner between $\sigma^Y$ and $\ti\sigma^Y$, then it extends uniquely to an intertwinner between $\sigma^\cO$ and $\ti\sigma^\cO$, and restricts to their respective smallest complete P-submodules. 
Hence, all of the above constructions are functorial. This gives the functor $M^Y$ from above. 
Moreover, it is clear that $M^Y$ is the inverse to the P-functor $\Pi^Y$ restricted to $\Mod^{\FD}_{\full}(\cP)$ up to a natural transformation.

In the $F$-case: atomic representations of P-dimension one do not have a unique extension to $\cO$ which is problematic. Indeed, the classification from \cite{Brothier-Wijesena24} showed if $m_1 = (1,0,\C)$ and $m_2 = (0,1,\C)$, then $\Pi^F(m_1)\simeq\Pi^F(m_2)$ while $\Pi^\cO(m_1)\not\simeq \Pi^\cO(m_2)$.
Subsequently, we consider $\Rep^{\FD+}(F)$, rather than $\Rep^{\FD}(F)$, and apply a similar procedure as above obtaining a functor $M^F : \Rep^{\FD+}(F) \rightarrow \Mod^{\FD+}_{\full}(\cP)$ and an equivalence of categories.

We can now define a binary operation of P-representations.

\begin{definition} \label{def:tensor-p-rep}
Consider $X=T,V,\cO$ (resp.~$X = F$) and $\pi,\ti\pi\in \Rep^{\FD}(X)$ (resp.~$\Rep^{\FD+}(F)$). Moreover, assume that the pair of P-modules $(M^X(\pi),M^X(\ti \pi)) $ satisfy \eqref{eq:condition-P}.
Define the following binary operation:
$$\pi\boxtimes \ti\pi:= \Pi^X( M^X(\pi)\boxtimes M^X(\ti\pi) ).$$
\end{definition}

Here are surprising counter-examples to properties that $\boxtimes$ may have expected to satisfy.

\begin{example} \label{ex:closed-equiv-class-counter}
We construct two P-modules $m,n$ satisfying that 
$$\Pi(m)\simeq \Pi(\bone) \text{ while } \Pi(n\boxtimes m) \not \simeq \Pi(n\boxtimes \bone) \text{ and } \dim_P(\Pi(n\boxtimes m))\neq \dim_P(\Pi(n\boxtimes \bone)).$$
	
Consider the three P-modules
	\begin{align*}
		m & = (A,B,\C^2)= \left( \frac{1}{4}
		\begin{pmatrix}
			\sqrt{2} & \sqrt{2}\\
			\sqrt{2} - 2 & \sqrt{2} + 2
		\end{pmatrix} \ , \ 
		\frac{1}{4}
		\begin{pmatrix}
			\sqrt{2} + 2 & \sqrt{2} - 2\\
			\sqrt{2} & \sqrt{2}
		\end{pmatrix} \ , \ \C^2 \right)\\
		n & = (1/2, \sqrt{3}/2, \C) \text{ and } \bone = (1/\sqrt{2}, 1/\sqrt{2}, \C).
	\end{align*}
	Note that $(1,1)^T$ is a common eigenvector of $A$ and $B$ with common eigenvalue $1/\sqrt 2$ while $(1,-1)^T$ is not a common eigenvector.
	We easily deduce that $\bone:=(1/\sqrt 2,1/\sqrt 2,\C)$ is the complete submodule of $m$ and thus $\Pi(\bone)\simeq \Pi(m)$. 
	Put $(\ti A, \ti B, \C^2) := n \boxtimes m$. Observe $\ti A$, $\ti B$ do not commute and thus do not share a common eigenvector. Hence, $n \boxtimes m$ is irreducible and $\dim_P(n \boxtimes m)=2$.
	Since $\dim_P(n\boxtimes \bone)=\dim_P(n)=1$ we deduce that $n\boxtimes m\not\simeq n\boxtimes \bone$ and in particular
	$\Pi(n\boxtimes \bone)\not\simeq \Pi(n\boxtimes m)$.
	
	In fact, this example demonstrates an even more surprising phenomena. Since $\boxtimes$ is associative and $(\Irr_{\diff}(1), \boxtimes)$ forms a group (see Section \ref{subsec:one-dim-ex}) we have that
	\[n^{-1} \boxtimes (n \boxtimes m) = m.\]
	This gives an example where the tensor product of a one-dimensional (irreducible) P-module ($n^{-1}$) with another irreducible P-module ($n \boxtimes m$) is not irreducible, nor even full. 
\end{example}

\subsection{Tensor category of Pythagorean representations.} \label{subsec:tens-cat-p-rep}
Recall the class of P-modules $\NI,\NII$ of Definition \ref{def:N-M} identified as full sub-categories of $\Mod(\cP)$.
Write $\Rep_{\NI}(X)$, $\Rep_{\NII}(X)$ for the full sub-categories of $\Rep(X), X=F,T,V,\cO$ with class of objects $\Pi^X(m)$ with $m \in \NI, \NII$, respectively.

\begin{proof}[Proof for Corollary \ref{lettercor:tensor}]
	The corollary follows from Theorem \ref{thm:tensor-category} and how the functors $\Pi^F, M^F$ restrict to an equivalence of categories between $\NII$, $\Rep_{\NII}(F)$ as explained in the previous subsection.
	The only detail that needs to be checked is whether these functors transport the monoidal structure on $(\NII, \boxtimes)$ to a monoidal structure on $(\Rep_{\NII}(F), \boxtimes)$ (i.e. verify the latter has the appropriate associativity isomorphisms). Though, this is a folklore result. Indeed, let $\eta$  be the natural transformation from $M^F \circ \Pi^F$ to $\id$ where the functors are restricted to $\NII$. Consider P-representations $\sigma_i \in \Rep_{\NII}(F)$ and define $M_i := M^F(\sigma_i) \in \NII$ for $i=1,2,3$. Then
	\[\id_{M_1} \otimes \eta_{M_2 \boxtimes M_3}^{-1} \circ a_{M_1, M_2, M_3} \circ \eta_{M_1 \boxtimes M_2} \otimes \id_{M_3} : (\sigma_1 \boxtimes \sigma_2) \boxtimes \sigma_3 \rightarrow \sigma_1 \boxtimes (\sigma_2 \boxtimes \sigma_3)\]
	defines the appropriate associativity isomorphisms for $(\Rep_{\NII}(F), \boxtimes)$, and $\Pi^F(\bone)$ satisfies the unit axiom. 
	Therefore, $(\Rep_{\NII}(F), \boxtimes)$ is a monoidal category and $M^F$ gives a monoidal equivalence of categories between $(\NII, \boxtimes)$ and $(\Rep_{\NII}(F), \boxtimes)$.
	A similar argument proves the other cases.
\end{proof}

\begin{remark} \label{rem:tensor-prod-prop}
	\begin{enumerate}
		\item The unit of $(\Rep_{\NI}(F),\boxtimes)$ is $\Pi^F(1/\sqrt{2},1/\sqrt{2},\C)$ which was proved in \cite[Proposition 6.2]{Brothier-Jones19} to be equivalent to the Koopman representation with respect to the classical action of $F$ on the unit interval.
		\item Given Definition \ref{def:tensor-p-rep} where we are required to restrict to $\Rep^{\FD+}(F)$ to define $\boxtimes$ for P-representations of $F$.
		Surprisingly, the class $\{\Pi^F(m):\ m\in \NI, \Pi^F(m)\in \Rep^{\FD+}(F)\}$ is not closed under $\boxtimes$, see Example \ref{ex:diff-tensor-to-atom}. This is why we are required to further restrict to $\NII$ for the $F$-case. 
	\end{enumerate}
	
\end{remark}

\section{Examples of the tensor product}\label{sec:examples}
	
In what follows, we will mostly explicitly work with full P-modules and freely identify them with P-representations using Theorem \ref{theo:classification}.
For this section, take $X=F,T,V,\cO$ unless specified otherwise. Moreover, we fix a co-ordinate system so that all P-modules are of the form $(A,B,\C^d)$ for some $d \geq 1$ where $A,B$ are $d$ by $d$ matrices. We write $(e_1, \dots, e_d)$ for the usual basis of $\C^d$. 

\subsection{Lie group action of $\Irr_{\diff}(1)$} \label{subsec:one-dim-ex}
We begin by studying tensor products of irreducible diffuse P-modules of dimension one 
$$\Irr_{\diff}(1)=\{(a,b,\C):\ a,b\in\C, |a|^2+|b|^2=1, a\neq 0\neq b\}.$$
Their associated P-representations contain many families of representations studied in the literature, to list a few; deformations of the Koopman representation of $V$ \cite{Garncarek12}, the Bernoulli representations of $\cO$ \cite{Olesen16}, and the Hausdorff representations of $\cO$ \cite{Mori-Suzuki-Watatani07} (see \cite[Section 6]{Brothier-Wijesena23} for details).

It is obvious by definition that $\boxtimes$ descends to an associative, commutative binary operation on $\Irr_{\diff}(1)$. Better yet, $\boxtimes$ endows $\Irr_{\diff}(1)$ with an abelian Lie group structure.
Recall the binary operation $\star$ from Definition \ref{def:star-operation}.

\begin{proposition} \label{prop:irr-lie-group}
	The set $\Irr_{\diff}(1)$ equipped with the tensor product $\boxtimes$ forms an abelian Lie group with identity $\bone = (1/\sqrt{2}, 1/\sqrt{2}, \C)$ and inverse
	\[(a, b, \C)^{-1} = (\vert b\vert e^{-i\Arg(a)}, \vert a \vert e^{-i\Arg(b)}, \C)\]
	where $\Arg$ is any argument function.
	Moreover, we have an isomorphism of Lie groups
	$$S_1\times S_1\times \R\to \Irr_{\diff}(1),\ (u,v,t)\mapsto \big(\frac{u}{\sqrt{e^t+1}} ,  \sqrt{\frac{e^t}{e^t+1}} v,\C\big).$$
	Hence, $(\Irr_{\diff}(1), \boxtimes)$ is a connected abelian Lie group of dimension $3$.
\end{proposition}

\begin{proof}
	This immediately follows from noting $((0,1), \star)$ forms an abelian Lie group (with inverse map $a \mapsto \sqrt{1-a^2}$) and is isomorphic to $(\R, +)$. The isomorphism is given by
	\[\varep : \R \rightarrow (0,1) : y \mapsto \frac{1}{\sqrt{e^y+1}}\]
	which can be verified by a simple calculation.
\end{proof}

\begin{corollary}
	The tensor product defines a Lie group action $\alpha_d:\Irr_{\diff}(1)\act\scrP_{\leq d}$ by
	\begin{align*}
		(a,b,\C) \cdot (A,B,\C^d) &:= (a,b,\C) \boxtimes (A,B,\C^d) = (aAK^{-1}, bBK^{-1}, \C^d) \\
		&= (e^{i\Arg(a)} U_A\circ  (\vert a \vert \star \vert A \vert ) , e^{i\Arg(b)} U_B \circ (\vert b \vert \star \vert B \vert) , \C^d)		
	\end{align*}
	where $\scrP_{\leq d}$ is the set of all P-modules over $\C^d$, $Arg$ is any argument function, and
	\[K = \sqrt{\vert a \vert^2\vert A \vert^2 + \vert b \vert^2 \vert B \vert^2}.\]
	Moreover, $\alpha_d$ restricts to an action on $\NI \cap \scrP_{\leq d}$ and $\NII \cap \scrP_{\leq d}$.
\end{corollary}

\begin{remark}
	In \cite{Brothier-Wijesena23} it was shown that the set $\Irr_{\diff}(d)$ of classes of irreducible diffuse P-modules of dimension $d$ forms a smooth manifold with real dimension $2d^2+1$. 
	However, $\alpha_d$ does not factorise into an action on the whole space $\Irr_{\diff}(d)$ as irreducibility (or even fullness) is not always preserved by $\alpha_d$, see Example \ref{ex:closed-equiv-class-counter}.
\end{remark}

\subsection{Tensor product of atomic and diffuse representations.} \label{subsec:atom-diff-tensor}
Recall from Sections \ref{sec:decomposition-irreducible} and \ref{subsec:diff-atom-rep} the decomposition of P-representations into diffuse and atomic parts. The atomic representations of the Cuntz algebra coincide with the purely atomic representations of \cite{DHJ-atomic15} which generalise the celebrated permutation representations of \cite{Bratteli-Jorgensen-99}. Moreover, the restriction of a sub-class to $F,T,V$ were studied in \cite{Barata-Pinto19, Araujo-Pinto22}.

Before looking at their tensor product, first we briefly describe the classification of atomic P-modules and P-representations of the Thompson groups, for details see \cite{Brothier-Wijesena24}.

{\bf Classification of atomic P-modules.}
Fix $d \geq 1$ and $W_d$ to be a set of representatives of prime binary strings of length $d$ modulo cyclic permutation. 
Given $w\in W_d$ we define the partial shift operators (indices are taken to be modulo $d$) on $\C^d$:
\[\begin{cases*}
	A_w(e_k) = e_{k+1} \textrm{ and } B_we_k = 0 \quad \textrm{if the $k$th digit of $w$ is $0$} \\
	A_w(e_k) = 0 \textrm{ and } B_we_k = e_{k+1} \quad \textrm{otherwise}.
\end{cases*}
\]
Given $z\in S_1$ set $D_z$ to be the diagonal matrix with last diagonal coefficient equal to $z$ and all other equal to $1$.
This data defines the atomic P-module
\[m_{w,z} := (A_wD_z,B_wD_z,\C^d).\]
The following is a combination of the main theorems of \cite{Brothier-Wijesena24}.

\begin{theorem} \label{thm:atom-classif}
	\begin{enumerate}
		\item Every irreducible atomic P-module is equivalent to $m_{w,z}$ for some $w \in W_d$ and $z \in S_1$, and conversely every $m_{w,z}$ is an irreducible atomic P-module.
		\item If $Y$ is a Thompson's group and $(Y,d)\neq (F,1)$, then 
		$$\Pi^Y(m_{w,z})\simeq \Ind_{Y_p}^Y\chi_{z}^p$$ 
		where $p=w^\infty$ in the periodic ray with period $w$, $Y_p=\{g\in Y:\ g(p)=p\}$ is the corresponding parabolic subgroup, and $\Ind_{Y_p}^Y\chi_{z}^p$ is the induction to $Y$ of
		$$\chi_{z}^p:Y_p\to S_1, \ z^{\log(2^d)(g'(p))}$$
		where $\log(2^d)$ is the logarithm in base $2^d$. 
	\end{enumerate}
\end{theorem}

\textbf{Tensor product between atomic and diffuse ones.}
Fix $r \geq 1$, a prime finite word $w \in W_r$, a unit scalar $z \in S_1$, an (irreducible) atomic P-module $m_a = (A_a, B_a, \C^r) := m_{w,z}$, and a diffuse P-module $m_d := (A_d, B_d, \C^s)$. Let $(e_1, \dots, e_r)$ and $(f_1, \dots, f_s)$ denote the usual standard basis of $\C^r$ and $\C^s$, respectively. 

Observe that either $A_a$ or $B_a$ will not be invertible. Hence, we shall assume both $A_d$ and $B_d$ are invertible so that we can consider the tensor product $m_a \boxtimes m_d$. 
Let 
$$A_d = V_A\vert A_d \vert \text{ and }B_d = V_B\vert B_d\vert$$ 
be the (unique) polar decomposition of $A_d$ and $B_d$. Furthermore, define the unitary operator $V = V_{x_r}V_{x_{r-1}}\dots V_{x_1}$ where $x_k = A$ if the $k$th digit of $w$ is $0$ otherwise $x_k = B$. Let $(\xi_i)_{i=1}^s$ be the eigenvectors of $V$ which unit eigenvalues $(\varphi_i)_{i=1}^s$.
The following result shows that atomic P-modules have strong absorption properties with respect to $\boxtimes$.

\begin{proposition} \label{prop:atomic-diffuse-tensor}
	Let $m_a$ and $m_d$ be atomic and diffuse P-modules, respectively, as defined above. Then the tensor product $m_a \boxtimes m_d$ is equivalent to the finite direct sum:
	\[\oplus_{i=1}^s m_{w, \varphi_iz}.\]
	Note that the prime word $w$ remains unchanged. Thus, the tensor product decomposes into $s$ number of irreducible atomic P-modules each with P-dimension $r$. Each resulting atomic P-module is obtained by shifting the phase of $m_a$ be an eigenvalue of $V$. Moreover, they are all inequivalent if and only if $V$ has distinct eigenvalues.
\end{proposition} 

\begin{proof}
	First, we consider a polar decomposition of $A_a$ and $B_a$. By the classification of atomic P-modules described above, we deduce that $\vert A_a \vert$, $\vert B_a \vert$ are diagonal matrices where either the $i$th diagonal entry of $\vert A_a \vert$, $\vert B_a \vert$ is $1$ and $0$, respectively, or vice versa. In particular, the concatenation of the diagonal entries of $\vert B_a\vert$ forms the binary word $w$. Let $U_a$ be the $d$ by $d$ generalised permutation matrix $U_a$ with $(i+1,i)$ entry equal to the $i$th diagonal entry of $D_z$ and all other entries equal to $0$. 
	The matrix $U_a$ is unitary and moreover the following are polar decompositions:
	\[A_a = U_a \vert A_a \vert \textrm{ and } B_a = U_a \vert B_a \vert.\]
	
	To compute $m_a \boxtimes m_d$ we shall use Proposition \ref{prop:alt-tensor-product-def}. 
	Consider $U_a\ot V_A$ that we interpret as a matrix with respect to the usual basis $(e_i\otimes f_j)_{i,j}$.
	If $U_a$ has $(i,j)$th entry $u_{i,j}$, then $U_a\ot V_A$ is a block matrix, with blocks of size $s$, where the $(i,j)$th block is $u_{i,j}V_A$. From the above description of $U_a$ we deduce that the $(i,j)$th block of $U_a\ot V_A$ is zero unless $j=i+1$ (modulo $r$), in which case it is then equal to a coefficient of $D_z$ times $V_A$. A similar description holds for $U_a \otimes V_B$. Hence, $U_a \otimes V_A + U_a \otimes V_B$ forms a permutation block matrix with blocks of size $s$ by $s$, and the product of all these blocks is equal to $zV$.
		
	Next we compute $\vert A_a \vert \star \vert A_d \vert$. 
	The numerator is a block diagonal matrix with blocks of size $s$ equal to $|A_d|$ or $0$. 
	The denominator is a block diagonal matrix with blocks of same size equal to $\vert A_d \vert^{-1}$ or $\sqrt{\id - \vert A_d \vert^2}^{-1/2}$. 
	Hence, $\vert A_a \vert \star \vert A_d \vert$ is a block diagonal matrix with $i$th block of size $s$ equal to $\id$ (resp.~$0$) when the $i$th diagonal entry of $\vert A_a \vert$ is $1$ (resp.~$0$).
	A similar expression exists for $\vert B_a \vert \star \vert B_d \vert$ and write 
	\[\ti A := (U_a \otimes V_A) \circ (\vert A_a \vert \star \vert A_d \vert),\ \ti B := (U_a \otimes V_B) \circ (\vert B_a \vert \star \vert B_d \vert).\]
	
	Now, define $\eta_i := e_1 \otimes \xi_i$ for $i = 1, \dots, s$. Since $\xi_i$ is an eigenvector of $V$ and by the above description, we deduce $x_rx_{r-1}\dots x_1\eta_i = \varphi_iz\cdot\eta_i$ where $x_k$ is defined as in the paragraph preceding the proposition. Thus, the vector $\eta_i$ generates the P-module $m_{w, \varphi_iz}$. The remainder of the proposition easily follows from the classification results of \cite{Brothier-Wijesena24}.	
\end{proof}

We obtain the Lie group action $\alpha_d$ from Section \ref{subsec:one-dim-ex} restricts to atomic P-modules.

\begin{corollary}
Consider $d\geq 1$, $g=(a,b,\C)\in \Irr_{\diff}(1)$, and an irreducible atomic P-module $m_{w,z}$ with $w\in W_d$ and $z\in S_1$. 
Set $\alpha=a/|a|$ and $\beta=b/|b|$.
We have that $g\boxtimes m_{w,z}$ is an irreducible atomic P-module equivalent to $m_{w,y}$ where 
$$y = \alpha^k\cdot \beta^{d-k}\cdot z$$
and $k$ is the number of $0$'s in $w$.
Additionally, if $Y=F,T,V$ and $(Y,d)\neq (F,1)$, then
$$\Pi^Y(g)\boxtimes \Ind_{Y_p}^Y \chi_z^p \simeq \Ind_{Y_p}^Y \chi_{y}^p \text{ where } p=w^\infty.$$
\end{corollary}

\textbf{Tensor product of diffuse P-modules.}
Surprisingly, the tensor product of two diffuse P-modules (that can be taken in $\NI$) is not necessarily diffuse.

\begin{example} \label{ex:diff-tensor-to-atom}
	Fix $\alpha, \ti\alpha \in S_1$ and $a, \ti a \in (0,1)$. Denote $P_i \in B(\C^2)$ to be the (orthogonal) projection onto $\C e_i$ for $i=1,2$. Let $S_{x_1,x_2}$ be the scaling matrix so that $e_i \mapsto x_ie_i$ for $x_1,x_2 \in \C$. Lastly, denote $R_\theta$ to be the rotation matrix for $\theta \in (0,\pi/2)$ so that $R_\theta$ has eigenvectors $(1,\ i)^T, (1,\ -i)^T$ with eigenvalues $e^{i\theta}, e^{-i\theta}$, respectively. Define the P-modules
	\begin{align*}
		m = (\alpha aP_2, R_\theta S_{1, a^{-1}}, \C^2),\ \ti m = (\ti\alpha \ti a P_1, R_\theta S_{\ti a^{-1}, 1}, \C^2)
	\end{align*}
	where the inverses $a^{-1},\ti a^{-1}$ are taken with respect to the group law on $((0,1), \star)$ (see Section \ref{subsec:one-dim-ex}). It is easy to verify that $m, \ti m$ are diffuse, irreducible and belong in $\NI$ (but not in $\NII$). Denote their tensor product by $n := m \boxtimes \ti m = (A,B,\C^4)$ where we have
	\[A = \textrm{diag}(0,\ 0,\ \alpha\ti\alpha(a\star\ti a),\ 0) \textrm{ and }
	B = R_\theta\otimes R_\theta \ \textrm{diag}(1,\ 1,\ (a \star \ti a)^{-1},\ 1).\]
	Note that $R_\theta \otimes R_\theta(\xi)=\xi$ where $\xi := e_1 \otimes e_1 - e_2 \otimes e_2$.
It follows that $\C \xi$ is a P-submodule of $n$ equivalent to the atomic P-module $(0,1,\C)$. In particular, $n$ is not diffuse and can be proven to not be atomic. It is easy to verify $n$ does not contain any other one-dimensional P-submodules. Thus, since $n$ is full (by Proposition \ref{prop:normal-full-mod}) we deduce that $\fK := (\C e\xi)^\perp$ must be a diffuse, irreducible P-submodule of $n$ of dimension $3$.
\end{example}

In \cite{Brothier-Wijesena22} it was shown that diffuse P-representations of the Thompson groups are $\NInd$ - they do not contain the induction of finite-dimensional representations. In strong contrast, atomic P-representations of the Thompson groups are finite direct sums of a specific class of monomial representations. Hence, the above example is interesting as it shows that even though $\Pi^Y(m)$, $\Pi^Y(\ti m)$ are $\NInd$, their tensor product does contain a monomial representation for $Y=F,T,V$. Even more surprisingly, by the classification results of \cite{Brothier-Wijesena24, Brothier-Wijesena23}, we have $\Pi^F(m)$, $\Pi^F(\ti m)$ are irreducible (infinite-dimensional) representations, yet their tensor product contains the trivial representation.

\subsection{Tensor product of diagonal and anti-diagonal matrices.}
\label{subsec:tensor-prod-diag-anti}
Consider
\[\cD_2 := \{(\begin{pmatrix} a_1&0\\0&a_2 \end{pmatrix}, \begin{pmatrix} 0&b_2\\b_1&0 \end{pmatrix}, \C^2) : a_i\neq 0\neq b_i,\ \vert a_i \vert^2 + \vert b_i \vert^2 = 1\} \subset \NII.\]
For $m = (A,B,\C^2) \in \cD_2$, using Theorem \ref{theo:classification} it can be easily verified that $\Pi^X(m)$ is irreducible if and only if $a_1 \neq a_2$. If $a_1 = a_2$ then $\Pi^X(m)$ decomposes into a direct sum $\Pi^X(a_1, \beta_1, \C) \oplus \Pi^X(a_2, \beta_2, \C)$ where $\beta_i$ are the square roots of $b_1b_2$.
Furthermore, let 
$$\ti m = (\ti A, \ti B, \C^2) = (\begin{pmatrix} \ti a_1&0\\0&\ti a_2 \end{pmatrix}, \begin{pmatrix} 0&\ti b_2\\ \ti b_1&0 \end{pmatrix}, \C^2)$$
be another P-module in $\cD_2$. Then $\Pi^X(m), \Pi^X(\ti m)$ are equivalent if and only if:
\begin{itemize}
	\item $a_1 = a_2 = \ti a_1 = \ti a_2$ and $b_1b_2$ = $\ti b_1 \ti b_2$; or
	\item $a_1 \neq a_2$ and $A = \ti A$ and $b_1 = \mu\ti b_1$, $b_2 = \bar\mu\ti b_2$ for $\mu \in S_1$; or
	\item $a_1 \neq a_2$ and $a_1 = \ti a_2,$ $a_2 = \ti a_1$ and $b_1 = \mu\ti b_2$, $b_2 = \bar\mu\ti b_1$ for $\mu \in S_1$.
\end{itemize}

Now we consider the tensor product between $m$ and $\ti m$.
Set $a_{ij} = a_i\ti a_j$, $b_{ij} = b_i\ti b_j$ and $k_{ij} = (\vert a_i\vert^2 \vert\ti a_j\vert^2 + \vert b_i\vert^2 \vert \ti b_j\vert^2)^{-1/2}$. Then $m \boxtimes \ti m = (D, E, \C^4)$ where $D$ is the diagonal matrix given by $D(e_i \otimes e_j) = a_{ij}k_{ij}e_i \otimes e_j$ and $E$ is the anti-diagonal matrix given by $E(e_i \otimes e_j) = b_{ij}k_{ij}e_{3-i} \otimes e_{3-j}$. 
This is unitary equivalent to the following direct sum:
$$	(\begin{pmatrix}
		a_{11}k_{11}&0\\
		0&a_{22}k_{22}
	\end{pmatrix},
	\begin{pmatrix}
		0&b_{22}k_{22}\\
		b_{11}k_{11}&0
	\end{pmatrix},
	\C^2) \oplus (
	\begin{pmatrix}
		a_{12}k_{12}&0\\
		0&a_{21}k_{21}
	\end{pmatrix},
	\begin{pmatrix}
		0&b_{21}k_{21}\\
		b_{12}k_{12}&0
	\end{pmatrix},
	\C^2).$$
Therefore, the tensor product will always decompose into two sub-representations each of P-dimension $2$. These components can then be further classified into irreducible classes via the above classification.
In particular, it is easy to construct examples of two irreducible inequivalent P-modules in $\cD_2$ whose tensor product decompose into irreducible components with P-dimension $(2,2)$, $(1,1,2)$ or $(1,1,1,1)$ (where each number in a tuple corresponds to the P-dimension of an irreducible component). 
As an example, the below tensor product gives four inequivalent P-modules each of P-dimension one. For brevity, we set $\lambda := 1/\sqrt{2}$ and let $\alpha \in S_1$.
\begin{align*}
	&(
	\begin{pmatrix}
		i\lambda & 0 \\  0 & -i\lambda 
	\end{pmatrix},
	\begin{pmatrix}
		0 & \lambda \\ \alpha\lambda & 0
	\end{pmatrix}
	) \boxtimes
	(
	\begin{pmatrix}
		i\lambda & 0 \\  0 & -i\lambda 
	\end{pmatrix},
	\begin{pmatrix}
		0 & \lambda \\ \bar\alpha\lambda & 0
	\end{pmatrix}
	) \\
	&\cong (-\lambda, \lambda) \oplus (-\lambda, -\lambda) \oplus (\lambda, \lambda) \oplus (\lambda, -\lambda). 
\end{align*}

More generally, we can define the set of P-modules $\cD_n$ for $n \geq 1$ consisting of P-modules $(A,B,\C^n)$ where $A,B$ are invertible with $A$ diagonal and $B$ anti-diagonal. 
Each P-module in $\cD_n$ has P-dimension $n$.
Further, define $\cD := \cup_{n\geq 1}\cD_n$. It can be seen that $\boxtimes$ restricts to a map from $\cD_m \times \cD_n$ to $\cD_{mn}$ and every P-module in $\cD$ decomposes into a finite direct sum of P-modules belonging in $\cD_1$ or $\cD_2$. Hence, the family of P-representations arising from $\cD$ is closed under taking: finite direct summands, P-submodules and tensor products. 
Moreover, observe by Remark \ref{rem:tensor-cat} that $\cD$ is a sub-tensor category of $\NII$.

\subsection{Tensor product of GP representations}
In \cite{Kawamura05} Kawamura introduced the rich class of so-called generalised permutation (GP) representations of $\cO$ (here we will also consider their restrictions to $F,T,V$). These representations are inspired as a generalisation of the permutation representations of \cite{Bratteli-Jorgensen-99}. 
We shall study the tensor product of two finite P-dimensional GP representations which we will show to be a finite direct sum of GP representations. 
We start by describing the P-modules that give rise to GP representations (see \cite[Section 6]{Brothier-Wijesena23} for details).

Fix $d \geq 1$ and a $d$-tuple $z = (z_1, \dots, z_d)$ with $z_k = (a_k,b_k) \in S_3$ (i.e. $\vert a_k\vert^2+\vert b_k\vert^2 = 1$). We refer to $z$ as the \textit{GP vector}. If any of the $z_k$ are not of the form $(a,0)$ or $(0,b)$ then we say the GP vector is \textit{invertible}. Define weighted shift operators $A_z,B_z \in B(\C^d)$ by
\[A_ze_i = a_{i}e_{i+1},\ B_ze_i = b_{i}e_{i+1}\]
where the indexes are taken to be modulo $d$. This gives a P-module (that we shall call a GP-module) and GP representation:
\[m_z := (A_z,B_z,\C^d),\ \kappa_z := \Pi^X(m_z).\]
Conversely, all finite P-dimensional GP-modules arise in this manner. The representation $\kappa_z$ of $X$ is irreducible if and only if $z$ is not periodic (there does not exist a tuple $y$ with $y \neq z$ such that $z$ is equal to multiple concatenations of $y$). Moreover, every GP-module is full and decomposes as a direct sum of irreducible GP-modules, and thus a similar statement holds for GP representations (however an arbitrary finite direct sum of GP representations is not necessarily GP). Two GP representations $\kappa_z, \kappa_{\ti z}$ are equivalent if and only if $z,\ti z$ are cyclic permutations of each other.

Let $\scrK$ to be the set of finite direct sums of $\kappa_z$ with $z$ \emph{invertible} and set $\scrK_d=\{m\in\scrK:\ \dim_P(m)=d\}.$ 
Observe that P-modules of $\scrK$ are full and identify $\scrK$ as a full sub-category of $\Mod_{\full}^{\FD}(\cP)$. Similarly, define $\Rep_{\scrK_d}(X), \Rep_{\scrK}(X)$ obtained via the P-functors $\Pi^X$.
They form monoidal categories that are rather disjoint from $\Rep_{\scrM}(X)$.

For two positive integers $m,n$ we write $\textrm{lcm}(m,n)$ (resp.~$\textrm{hcf}(m,n)$) to be the lowest common multiple (resp.~highest common factor) of $m$ and $n$.

\begin{theorem}
	Fix $r,s \geq 1$ and set $l := \textrm{lcm}(r,s)$, $h : = \textrm{hcf}(r,s)$. Consider aperiodic invertible GP vectors $z, \ti z$ of length $r,s$, respectively. The following assertions are true.
	\begin{enumerate}
		\item The product $\kappa_z \boxtimes \kappa_{\ti z}$ is equivalent to a direct sum of $h$ (possibly reducible) GP-modules each of P-dimension $l$.
		\item The product $\boxtimes$ maps $\scrK_r \times \scrK_s$ to $\scrK_{rs}$.
		\item The sub-categories $(\scrK, \boxtimes)$, $(\Rep_{\scrK}(X), \boxtimes)$ are pivotal unitary tensor categories.
	\end{enumerate}
\end{theorem}

\begin{proof}
	First, observe that when the GP vector $z$ is invertible, then $A_z,B_z$ are invertible. Hence, the product $\boxtimes$ between two P-modules in $\scrK$ is well-defined.
	The second statement in the theorem follows from the first. Similarly, the third statement follows from the first and from the third item of Remark \ref{rem:tensor-cat}. Hence, we only need to show the first statement. Its proof follows a similar argument of computing various polar decompositions as in Proposition \ref{prop:atomic-diffuse-tensor}. Hence, we shall be brief in our reasoning.
	
	Let $z = ((a_1, b_1), \dots, (a_r, b_r))$, $\ti z = ((\ti a_1, \ti b_1), \dots, (\ti a_s, \ti b_s))$, and let $\alpha_i \in S_1$ be such that $a_i = \alpha_i\vert a_i \vert$, similarly defining $\beta_i, \ti \alpha_i, \ti \beta_i$. The polar decomposition of $A_z$ consists of the unitary generalised permutation matrix with $(i+1,i)$ entry being $\alpha_{i}$, and a positive diagonal matrix with diagonal entries $\vert a_i \vert$. We have a similar description for $B_z, \ti A_z, \ti B_z$.
	
	Set $(D,E,\C^{r} \otimes \C^s) := m_z \boxtimes m_{\ti z}$. Then from the above descriptions we deduce that 
	\[D(e_i \otimes e_j) = \alpha_i\alpha_j \vert a_i\vert \star \vert a_j\vert \cdot e_{i+1} \otimes e_{j+1},\ 
	E(e_i \otimes e_j) = \beta_i\beta_j \vert b_i\vert \star \vert b_j\vert \cdot e_{i+1} \otimes e_{j+1}\]
	where the indices in the first (resp.~second) component of the tensors is taken modulo $r$ (resp.~$s$). From there, it is easy to conclude for any vector $e_i \otimes e_j$ its span under words in $D,E$ forms a GP-module $\fK$ of dimension $l$ (and thus of P-dimension $l$ since every GP-module is full). This can be repeated for any other vector $e_m \otimes e_n \in \fK^\perp$ to yield another GP-module of dimension $l$. This concludes the proof of the theorem.
\end{proof}

Continuing to use the same notation as above, the proof demonstrates there is an easy way to determine the decomposition of $\kappa_z  \boxtimes \kappa_{\ti z}$. First take the vector $z$ and concatenate itself $d := l/r$ times so that it has length $l$. Similarly concatenate $\ti z$ by $e := l/s$ times. Each of their components $z_i, \ti z_j$ can be identified with P-modules in $\Irr_{\diff}(1)$ (hence, $z^d, z^e$ can be viewed as being in the Lie group $\Irr_{\diff}(1)^l \cong S_1^{2l} \times \R^l$ where the group multiplication is taking $\boxtimes$ component wise). 
Take the component wise tensor product $\boxtimes$ between the vectors $z^d$, $z^e$ and denote the resulting GP vector by $y_1$. Now cyclically permute $z^e$ by one position and again take component wise $\boxtimes$ with $z_d$ to give a GP vector $y_2$. Continue this $h$ times so that the GP vectors $y_1, \dots, y_h$ are obtained. Then we have
\[\kappa_z \boxtimes \kappa_{\ti z} \cong \oplus_{k=1}^h \kappa_{y_k}.\]
We now deduce explicit fusion rules.

\begin{example}
\begin{enumerate}
\item
Let $m,n$ be any two relatively prime positive integers. Take any two aperiodic invertible GP vectors $z,\ti z$ with length $m,n$, respectively. 
Their GP-representations $\kappa_z,\kappa_{\ti z}$ of $X$ are irreducible, and so is their tensor product $\kappa_z \boxtimes \kappa_{\ti z}$.
\item
Let $z = (z_1, z_2)$ be any aperiodic invertible GP vector (hence, $z_1,z_2 \in \Irr_{\diff}(1)$ and $z_1 \neq z_2$). Take $\ti z := z^{-1} = (z_1^{-1}, z_2^{-1})$. Then $\kappa_z \boxtimes \kappa_{\ti z} = \kappa_{y_1} \oplus \kappa_{y_2}$ where
\[y_1 = (\bone, \bone) \textrm{ and } y_2 = (z_1 \boxtimes z_2^{-1}, z_2 \boxtimes z_1^{-1}).\]
Hence, this gives an example of two irreducible inequivalent GP-representations whose tensor contains two copies of the unit representation $\Pi^X(\bone)$ (recall $\bone = (1/\sqrt{2}, 1/\sqrt{2}, \C)$ and $\Pi^F(\bone)$ is the classical Koopman representation of $F$). 
\item
If $y$ is an invertible GP vector of length four and there exists invertible GP vectors $z, \ti z$ of length smaller than four such that $\kappa_y \cong \kappa_z \boxtimes \kappa_{\ti z}$, then $z, \ti z$ must both have length two. This implies that $\kappa_y$ is necessarily reducible. This shows that if $x$ is an aperiodic invertible GP vector of length four, then $\kappa_x$ cannot be expressed as a direct sum or tensor product of GP-representations in $\scrK$ of lower P-dimension. 
		
For a less trivial example, consider invertible GP vectors $z,\ti z$ of length two and three, respectively. Let $\kappa_y = \kappa_z \boxtimes \kappa_{\ti z}$ where $y = (y_1, \dots, y_6)$. Observe that necessarily $y_1 \boxtimes y_3^{-1} = y_4 \boxtimes y_6^{-1}$. Hence, if $x$ is an aperiodic invertible GP vector of length six which does not satisfy the above relation (up to cyclic permutation) then the same conclusion can be made for $\kappa_x$ as in the preceding paragraph.
\item
Recall the Lie group action $\alpha_d$ of $G := \Irr_{\diff}(1)$ from Section \ref{subsec:one-dim-ex}. Identifying $\scrK_d$ with $G^d$, it is immediate $\alpha_d$ restricts to an action on $\scrK_d$.
If $z$ is aperiodic then so is $g\cdot z$. Moreover, if $z, \ti z$ are not cyclic permutations of each other then same is true for $g \cdot z, g \cdot \ti z$. Thus, $\alpha_d$ restricts to a free action on the irreducible classes of P-modules in $\scrK_d$.  
\end{enumerate}
\end{example}

\section{Comparison with other binary operations of representations}\label{sec:Kawamura}
Kawamura has defined a binary operation $\otimes_\varphi$ on representations of all the Cuntz algebras $\cO_2,\cO_3, \dots$ where $\cO_n$ denotes the usual Cuntz algebra generated by $n$ isometries \cite{Kawamura07}. 
The operation $\otimes_\varphi$ is constructed from the family of *-algebra embeddings
$$\varphi_{n,m}:\cO_{nm}\to \cO_{n}\otimes \cO_m,\  s_{m(i-1)+j} \mapsto t_i\ot u_j$$
where $s_k,t_i,u_j$ are the usual isometries generating $\cO_{nm},\cO_n,\cO_m$, respectively. This is well-defined thanks to the universality of $\cO_{nm}.$
This yields a binary operation
$$\otimes_\varphi:\Rep(\cO_n)\times \Rep(\cO_m)\to \Rep(\cO_{nm}),\ (\pi, \sigma)\mapsto (\pi\ot \sigma)\circ \varphi_{n,m}.$$

The operation $\otimes_\varphi$ is well-defined on equivalence classes of representations, associative (thanks to the specific choice of the bijection between generators) and distributive with respect to the direct sum. However, it is not symmetric (Kawamura provides an elementary counter-example using permutation representations \cite{Bratteli-Jorgensen-99}).

Similarly, for each $n\geq 2$ we have a Pythagorean algebra $\cP_n$ which is the universal $C^*$-algebra with generators $a_1,\dots,a_n$ and satisfying the relation $a_1^*a_1+\dots + a_n^*a_n=1$ \cite{Brothier-Jones19}.
Now, just as in the $n=2$ case, we can define a Pythagorean functor $$\Pi_n:\Rep(\cP_n)\to \Rep(\cO_n).$$

One can show that our Pythagorean process is compatible with the construction of Kawamura in the following sense.
Define $\mu_{n,m}:\cP_{nm}\to \cP_n\ot \cP_m$ similarly to Kawamura by using the generators of the various Pythagorean algebras.
This yields a binary operation $\otimes_\mu$ on the family of all representations of all Pythagorean algebras.
Now, composing before by $\Pi_{nm}$ or after by $\Pi_n\ot \Pi_m$ yields a commutative diagram up to isomorphism.

\newcommand{\etalchar}[1]{$^{#1}$}

\end{document}